\documentclass[12pt]{iopart}
\usepackage{iopams}

\usepackage{graphicx}
\newtheorem{theorem}{Theorem}[section]
\newtheorem{lemma}[theorem]{Lemma}

\newtheorem{corollary}[theorem]{Corollary}
\newtheorem{definition}[theorem]{Definition}
\newenvironment{proof}[1][Proof]{\begin{trivlist}
\item[\hskip \labelsep {\bfseries #1}]}{\hfill$\Box$\end{trivlist}}
\usepackage{color} %Added by Anne
\newcommand{\bzero}{\mathbf{0}}
\newcommand{\Map}{F_n} %{M}
\newcommand{\tW}{\widetilde{W}}

\begin{document}

\title[Creation of discontinuities in circle maps
]{Creation of discontinuities in circle maps}

\author{Gianne Derks$^1$, Paul A. Glendinning$^2$ \& Anne C. Skeldon$^1$}

\address{$^1$ Department of Mathematics, University of Surrey,
  Guildford GU2 7XH, UK \\
$^2$ School of Mathematics, University of Manchester, Oxford Road, Manchester M13 9PL, UK }

\ead{g.derks@surrey.ac.uk, paul.glendinning@manchester.ac.uk, a.skeldon@surrey.ac.uk}
%\vspace{10pt}
%\begin{indented}
%\item[]August 2017
%\end{indented}

\begin{abstract}
Circle maps frequently arise in mathematical models of physical or biological systems.
Motivated by Cherry flows and `threshold' systems
such as integrate and fire neuronal models, models of cardiac
arrhythmias, and models of sleep/wake regulation, we consider
how structural transitions in circle maps occur. In particular we describe how maps
evolve near the creation of a discontinuity.

We show that the natural way to create discontinuities in the
  maps associated with
both threshold systems and Cherry flows results in a square root
singularity in the map as the discontinuity is approached from either
one or both sides.  We analyse the generic properties of maps with
  gaps displaying
this local behaviour, showing how border collisions and saddle-node
bifurcations are interspersed. This highlights how the Arnold tongue
picture for tongues bordered by saddle-node bifurcations is amended
once gaps are present.

For the threshold models we also show that a loss of injectivity
naturally results in the creation of multiple gaps
giving rise to a novel codimension two bifurcation.

%\pg{Simplified title and shorter abstract emphasising that our work is around the discontinuity even when it introduces non-monotonicity. Added keywords and ams classification.}
\end{abstract}


%Uncomment for keywords
\vspace{2pc}

\maketitle

%%%%%%%%%%%%%%%%%%%%%%%%%%%%%%%%%%%%%%%%%%%%%%%%%%%%%%%%%%%%%%%%%%%%%%%%%%%%%%%%%%%%%%%%%
% Introduction
%%%%%%%%%%%%%%%%%%%%%%%%%%%%%%%%%%%%%%%%%%%%%%%%%%%%%%%%%%%%%%%%%%%%%%%%%%%%%%%%%%%%%%%%%
%\input{Intro}
%\pg{Start with maps and try to unify the introduction and results.}
%\gd{Removed type I-IV for maps as they were not really used and were
%  confusing as we also have type I-II border collisions}

\section{Introduction}\label{section:intro}
Degree one circle maps $f:{\mathbb S}^1\to {\mathbb S}^1$ are
described by real functions $F:{\mathbb R}\to {\mathbb R}$ with
$F(x+1)=F(x)+1$ and $f(x)=F(x)$ modulo 1. These maps arise naturally
in many situations and $F$ may be injective (or not) and continuous
(or not) leading to four different types of map:  $f$ is injective
(monotonic) and continuous, %(type I),
the classic case; non-injective and continuous, %(type II),
\cite{Glass_1979};
injective and discontinuous, %(type III),
\cite{Glass_1991}; and
non-injective and discontinuous, %(type IV),
\cite{Skeldon_2014}.

In many applications the type of map is fixed. However, for maps
derived from the classes of models considered below changes of type
occur naturally as parameters vary.  Close to the transitions between
types, the maps have a well-defined structure. This structure in turn
changes some dynamical properties of the systems (note that the
transition between type is not necessarily a bifurcation \emph{per
  se}). These transitions and their consequences are the subject of
this paper.

Our motivation comes from two classes of models. The first arises in
many biological models including models of cardiac arrhythmias (see
\cite{Arnold_1991} and \cite{Glass_1991} and the references therein),
neuronal models \cite{Bressloff_1990,Glendinning_1995}, and the two
process model of sleep-wake regulation \cite{Borbely_1982,Daan_1984}.
We refer to these examples as `threshold systems', since in each case
a variable of interest increases until it hits an upper threshold,
decreases until it hits a lower threshold and then repeats.  Some
typical examples are shown in Fig.\ref{fig:thresholdmodels}.  If the
thresholds are periodic with the same period, then each system can be
represented by a circle map \cite{Glass_1991, Nakao_1998,
  Skeldon_2014} and the resulting observed phenomena include
phase-locked solutions, `period-adding',  period-doubling and
chaos.

The second class of models has found application in problems of breathing
rhythms \cite{Baesens_2013a} and arises naturally in coupled oscillator
problems at appropriate parameter values \cite{Baesens_2013b}. The
initial model is a flow on the torus. If there are no stationary
points and a global cross section (a Poincar\'e flow), then the return
map on this section is a continuous monotonic circle map. If, as
parameters are varied, a pair of stationary points is created by
saddle-node bifurcations then the resulting flow is known as a Cherry
flow \cite{Cherry_1938}. These can generate return maps which have
either discontinuities or regions where the map is not defined
\cite{Palis_1982}.

\begin{figure}[tbhp]
\centering{\includegraphics[scale=1.2]{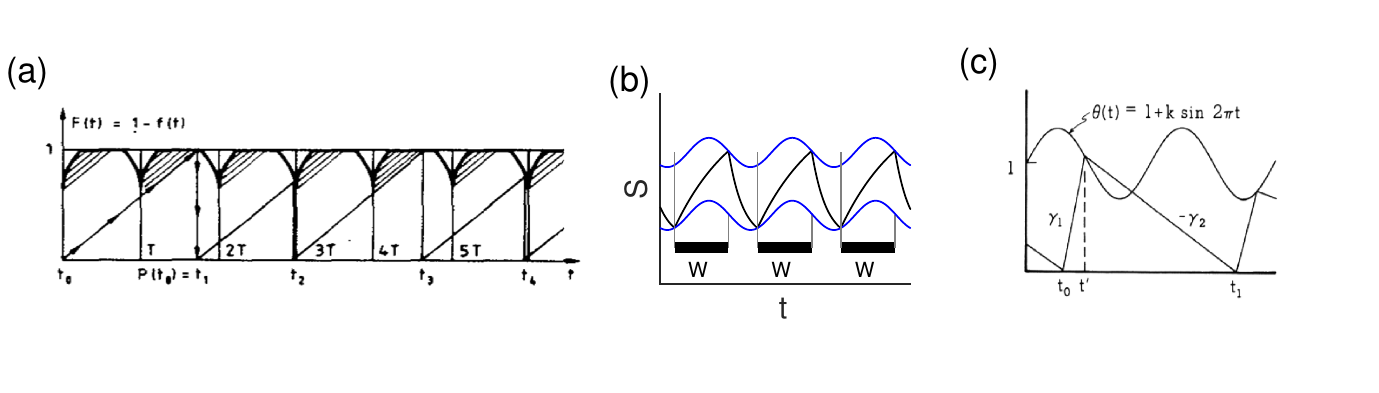}}
\caption{
(a) A model of cardiac arrhythmias,
attributed to Gel'fand and Tsetlin.
Reprinted from \cite{Arnold_1991} with the permission of AIP Publishing.
(b) The two-process model of sleep-wake regulation, sketch based
  on the model in \cite{Daan_1984}. (c)
An integrate and fire model.
Reprinted from \cite{Glass_1991} with the permission of AIP Publishing.
This model will be described in detail in
section~\ref{sect:STS} and be called the sinusoidal threshold system (STS).
}
\label{fig:thresholdmodels}
\end{figure}

There is a vast literature on  %type I
circle maps which are both continuous and monotonic (e.g. \cite{Katok_1995}). All points have the same
rotation number (average rate of rotation) under iteration by such
maps. If the rotation number is rational then solutions tend to
periodic orbits. Whilst if the rotation number is irrational there are no periodic
solutions and, if the map is sufficiently smooth (e.g. $C^2$), all
orbits are dense in the circle. Deeper results about the smoothness of
conjugacies to rigid rotations for maps with irrational rotation
numbers were developed by Herman \cite{Herman_1979}, and led to many
technical results in this direction \cite{vStrien_1993}.  For typical
families of circle maps the rotation number takes rational values
  on closed intervals of parameters, % there are open sets of parameter space in
% which the rotation number is a constant rational value,
this is called mode-locking and the regions of parameter space with a
given rational rotation number is an Arnold tongue. The Arnold tongues
are bounded by saddle-node bifurcations.

If the circle map is continuous but not monotonic %(type II)
then the rotation number is replaced by a rotation interval
\cite{Ito_1981}. Many properties can be understood using classic
results for maps of the interval and the transition
from %type I to type II
continuous and monotonic circle maps to continuous and non-monotonic
maps has been described in some detail, including the transitions to
chaos which involve different sequences of period-doubling
bifurcations \cite{Boyland_1986, Mackay_1986,MacKay_1988}.

The injective and discontinuous circle maps %type III
arise in a number of contexts and basic results such as the existence of a well-defined rotation number can be found in \cite{Keener_1980,Rhodes_1986,Rhodes_1991}.  The
review paper \cite{Granados_2017} gives a thorough summary of the
current literature on these monotonic circle maps with gaps, i.e. intervals with no 
pre-images, and shows how
maps of the real line with gaps can be framed as circle maps. This is
particularly important because it highlights how many results known
from the study of circle maps have been rediscovered in the context of
maps of the real line.

The non-injective discontinuous %type IV
circle maps can be divided into sub-classes. If the continuous
branches are increasing then this includes the Lorenz maps and again a
lot is understood, e.g. \cite{Glendinning_1993}. Less is known about
the details of the dynamics of %type IV
non-injective discontinuous maps in general (although the techniques of kneading theory do apply), partly because it is less clear what results would be useful without further context.

Both threshold systems and the transition from Poincar\'e flows to Cherry flows
provide natural settings to consider the transition from continuous
circle maps to piecewise continuous circle maps with discontinuities.  In each
case, specific properties of the original dynamical system induce
transitions between different circle map types. For the transition to discontinuous maps in both of these cases, we show
that the derivative of the map is singular on at least one side
of the discontinuity and derive scaling results.

Although the derivative becomes singular, derivatives may be large for
a small neighbourhood of the singular value and so can be difficult to resolve numerically.
Nevertheless, we show that the presence
of the singularity is essential to understanding how a continuous
circle map with phase-locked regions bounded by saddle-node bifurcations
transitions to a circle map with a gap, with creation/destruction of periodic solutions via border collisions and saddle-node bifurcations.

The layout of the paper is as follows. In
section~\ref{sect:threshold} we define smooth threshold systems and
the associated circle maps. An extension of the standard Arnold map is
presented as an example. In section~\ref{sect:tangencies}, we discuss
the creation of gaps in smooth threshold systems, deriving the typical
scalings for the gap size, and showing that the map to one side of the
gap has a square root singularity. In section~\ref{sect:squareroot} we
consider a general form for a piecewise continuous map with a gap on
which one side the map has a square root singularity and discuss two
examples. In section~\ref{sect:nonmonotonic} we discuss the creation
of non-monotonicity in threshold systems and how this can result in
multiple gaps and lead to codimension two bifurcations.  In
section~\ref{sect:Cherryflow} we consider Cherry flows and discuss the
creation of a discontinuity, showing that a finite gap is created
instantaneously and that both sides of the gap have a square root
singularity.  We end with a short discussion.

%%%%%%%%%%%%%%%%%%%%%%%%%%%%%%%%%%%%%%%%%%%%%%%%%%%%%%%%%%%%%%%%%%%%%%%%%%%%%%%%%%%%%%%%%
% Threshold models
%\input{Threshold_models}

\section{Threshold systems}\label{sect:threshold}
The essential feature of a threshold system is that there is a
dependent variable which increases until it hits an upper threshold,
decreases until it hits a lower threshold, and then repeats. The
following definition formalises this idea in the case of smooth
thresholds which provide the generic cases described later in this
paper. In this definition, $x$ represents the independent time-like
variable, and for any flow $\phi_x:{\mathbb R}\to{\mathbb R}$,
$\phi_x$ depends smoothly on the independent variable $x$, $\phi_0$ is
the identity and for all $r,s\in{\mathbb R}$,
$\phi_r\circ\phi_s=\phi_{r+s}$.

\begin{definition}\label{def:TS}
%\Def{Threshold system}
  A smooth threshold system is a pair of %autonomous
  flows $\phi_x$ and $\psi_x$, the up flow and down flow respectively,
  and an upper threshold and a lower threshold such that
%\pg{I've simplified this
%  because if both thresholds are nasty, like the comb, then the
%  resulting return map is not a map as it can be one to many in the
%  independent variable -- it needs more information to identify which
%  trajectory goes to which.}
%\gd{Added  up flow is increasing and  down flow is decreasing to make life
%  easier later.}
\begin{enumerate}
  \item[(i)] The up flow is strictly increasing %($\frac{d}{dx}\phi_x >0$)
    and the down flow is strictly decreasing. %($\frac{d}{dx}\psi_x <0$).
\item[(ii)] The upper and lower thresholds are the graphs of smooth real functions $y=h(x)$ and $y=g(x)$ with period one respectively such that for all $x\in {\mathbb R}$
\begin{equation}
g(x) < h(x). \label{eq:threshold_ordering}
\end{equation}
\item[(iii)] Starting from the lower threshold, the up flow
    reaches the upper threshold in finite time and vice
    versa. Formally, if $y_0= g(x_0)$ then there exists $\tau >0$
  such that $\phi_{\tau}(y_0)=h(x_0+\tau )$, and if $y_0 =h(x_0)$ then
  there exists $\tau^\prime>0$ such that
  $\psi_{\tau^\prime}(y_0) =g(x_0+\tau^\prime )$.
%The up flow is increasing on $G$,  i.e. $\left.\frac{d}{dx}\phi_x(y)\right|_{x_0}>0$ for all
%      $u\in \mathbb{R}$, and if $(t_0,y_0)\in G$,
%      then} %($\phi(x,y,0) \in G$ and
%    there exists $\tau>0$ such that
%    $(t_0+\tau,\phi_{\tau}(y_0)) \in H$ and
%    $(t_0+t,\phi_{t}(y_0)) \not \in H$ for $0\leq t < \tau$.
%  \item[(iv)] A down flow {$\psi_{t}:\mathbb{R}\to\mathbb{R}$,
%      which smoothly depends on~$t$ and is}
%  % $\psi(x,y,t)$
%  such that {$\psi_{0}$ is the identity,
%    $\psi_t\circ\psi_s =\psi_{t+s}$,
%    $\left.\frac{d}{dt}\psi_t(y)\right|_{t=0}<0$ for all
%      $u\in \mathbb{R}$,  and if $(t_0,y_0)\in H$,
%    then} %$\psi(x,y,0) \in F$ and
%  there exists $\tau>0$ such that $(t_0+\tau,\psi_{\tau}(y_0)) \in G$ and
%  $(t_0+t,\psi_{t}(y_0)) \not \in G$ for $0\leq t < \tau$.
\end{enumerate}
\end{definition}

To determine the dynamics of a threshold system consider an initial
condition on the upper threshold, $(x_n,y_n)$ with $y_n=h(x_n)$. By
property (iii) there is a smallest $\tau_n >0$ such that
$({x}_n+\tau_n,\tilde{y}_n)$ is on the lower threshold with
$\tilde{y}_n=\psi_{\tau_n} (y_n)=g(x_n+\tau_n)$. Now using the up flow
part of property (iii) there exists a smallest $\tau^\prime_n >0$ such
that
\[
y_{n+1}%=h(x_n+\tau_n+\tau^\prime_n,
% \phi_{\tau^\prime_n}(\tilde{y}_n)).
:=\phi_{\tau^\prime_n}(\tilde{y}_n) = h(x_n+\tau_n+\tau^\prime_n).
\]
Thus starting at $x_n$ on the upper boundary, the trajectory returns
to the upper boundary at time $x_n+\tau_n +\tau_n^\prime$, generating a
map
\[
x_{n+1}=F(x_n)=x_n+\tau_n+\tau^\prime_n.
\]
If the trajectory had started at the upper boundary with $x$-coordinate $x_n+1$ then the periodicity of $g$ and $h$ implies that the return to the upper boundary would be at $x_n+1+\tau_n+\tau^\prime_n$ and so $F(x+1)=F(x)$ thus $F$ is the lift of a degree one circle map.

In sections~\ref{sect:tangencies} and \ref{sect:nonmonotonic} {we
  will show that monotonicity and continuity of the circle map of a smooth
  threshold system is related to the absence of tangencies between the
  thresholds and %up and/or down
  flows.}
% for suitable conditions on the thresholds as compared with the
% up/down flows, the circle map of a threshold system will be
% monotonic and continuous.
%
% As noted earlier, the dynamics of monotonic, continuous circle maps are well-understood:
% a rotation number can be defined; if the maps depend on a
%   parameter in a monotonic and continuous way, then variation of this
% parameter varies the rotation number, %varies continuously
% following a Devil's staircase structure with almost all rotation
%   numbers being rational (the mode-locked regions).
%
%This results in a Farey sequence of
%phase-locked regions bounded by saddle-node
%bifurcations. Interspersing the phase-locked regions are quasiperiodic
%solutions whose orbits densely fill the circle{, if the map is
%smooth ($C^2$).}
%
Throughout this paper we illustrate our findings with the standard
example of a sinusoidal threshold, sketched in
Fig.~\ref{fig:thresholdmodels}(c) and described below. The section
below also contains the Arnold tongue structure for this example.
% We illustrate this Arnold tongue structure and subsequent results

\subsection{Example: the Sinusoidal Threshold System
  (STS)}\label{sect:STS}
% We will refer to the example of
% Fig.~\ref{fig:thresholdmodels}(d).
% as the Sinusoidal Threshold System (STS).
We believe that this dynamical system first appeared {explicitly} in a
paper by Glass {and Belair in 1986~\cite{Glass_1986}}.  {In a 1991
  paper~\cite{Glass_1991},} Glass refers to it as the Gel'fand and
Tsetlin integrate and fire model, though he was unable to locate a
reference. He notes that it is also studied by Arnold in his {1959}
  thesis, an excerpt of which is in~\cite{Arnold_1991}, though no
  explicit model is given.
  % This excerpt doesn't give an explicit model.  , and sometimes the
%   model}
%   % considered to be that
%   shown in Fig.~\ref{fig:thresholdmodels}(a) (reproduced
%     from~\cite{Arnold_1991}) is called the Gel'fand and Tsetlin model.
%   %that was studied by Arnold in his   thesis.
% As we will see below, the
% differences in the smoothness of the thresholds between (a) and (d)
% has profound implications for understanding the dynamics.
  Although the STS appears in \cite{Glass_1991}, the
  dynamics were not fully analysed. In ~\cite{Glass_1986}, three
  special cases are considered: infinite slope for the down flow
  (reset dynamics), equal rates for the up and down flow, and infinite
  slope for the up flow.
  % ($\gamma_2\to\infty$, $\gamma_1=\gamma_2$,
  % $\gamma_1\to\infty$.
  The first and last case seem to investigated in detail
  in various papers, but middle one is referred to as ``hope to
  investigate later''. The case with the infinite slope in the down
  flow is considered by Winfree in~\cite{Winfree_1980} in the context of
  the  entrainment of circadian rhythms. %(chapter 20, fig 1, pp 595,
The STS can also be thought of as a
simplified form of the two process model of sleep-wake regulation~\cite{Borbely_1982,Daan_1984}.

For the STS,  the upper and lower thresholds are given by the functions
\begin{eqnarray}
y  & =  & h(x) = \beta  + \frac{\alpha}{2 \pi} \left ( 1 + \sin 2 \pi x \right ), \nonumber \\
y  & = & g(x) = 0, \label{eq:toymap1}
\end{eqnarray}
respectively, with $\alpha >0$ and $\beta >0$.
The up and down flows are linear functions as shown
in Fig.~\ref{fig:thresholdmodels}(c) %, so after rescaling time
$\phi_{x}(y_0)=y_0+\gamma \,x$, $\gamma \ge 0$, (i.e. the solution to $\frac{dy}{dx}=\gamma$) and $\psi_{x}(y_0)=y_0-x$ (the solution to $\frac{dy}{dx}=-1$).
This system is equivalent to the system in~\cite{Glass_1991} with
  a rescaling of the parameters.
%\pg{Rewritten to emphasise connections to known systems}

% Although the dynamics of a threshold system was described by starting at the lower threshold, it is of course equally possible to start at the upper threshold, and for this example this choice makes connections with some standard models clearer.

Suppose that at $x_n$ the system is on the upper threshold, with $y_n=h(x_n)$.  Then the trajectory of the down flow
will reach the lower threshold after additional time $\tau_n$, where $\tau_n=h(x_n)$, and the new value of the independent variable is $\tilde{x}_n=x_n+\tau_n$ with $y$-coordinate equal to zero. The time taken to return to the upper threshold using the up flow is $\tau_n^\prime$ which is determined implicitly by
\begin{equation}
\gamma \tau^\prime_n= h(x_n+\tau_n+\tau^\prime_n).  \label{eq:toymap3}
\end{equation}
If (\ref{eq:toymap3}) has multiple positive solutions then the smallest possible $\tau^\prime$ is chosen. The return map to the upper threshold is therefore $x_{n+1}=F(x_n)=x_n+\tau_n+\tau^\prime_n$ or (as an implicit difference equation)
\begin{equation}
x_{n+1}=x_n+h(x_n)+\frac{1}{\gamma}h(x_{n+1}). \label{eq:implicitF}
\end{equation}
An immediate consequence is that as $\gamma\to \infty$ (a classic
reset) this reduces to the classic Arnold (sine) circle map
\cite{Arnold_1991,Glass_1991} which has been extensively studied.

The gradient of the map determines properties such as
monotonicity. Implicit differentiation of (\ref{eq:implicitF}) gives
\begin{equation}
  \frac{d x_{n+1}}{d x_n} =
   \frac{1 + h'(x_n)}{1 - \frac{1}{\gamma}h'(x_{n+1})} .
\label{eq:mapgradient}
\end{equation}
Since $h^\prime (x)= \alpha\cos (2\pi x)$, the numerator of
(\ref{eq:mapgradient}) is always positive if $\alpha <1$ (recall
$\alpha$ is positive) and the denominator is positive %non-zero
provided $\alpha <\gamma$.  In particular, the map is monotonic and
continuous if $\alpha < \min (1,\gamma)$.  The derivative becomes
singular when the up flow becomes tangent to the upper threshold.  In
section~\ref{sect:tangencies} we will show that this a generic feature
in maps generated by threshold systems and describe the generic
development of a discontinuity and singular derivative in the map.
Similarly,~\eref{eq:mapgradient} shows that the derivative vanishes
when the down flow becomes tangent to the upper boundary.  In
section~\ref{sect:nonmonotonic} we will show that such a tangency
generates non-monotonicity in the map generated by a threshold
system. Thus the STS example illustrates
%  discuss the impact of the derivative becoming singular for the
the generic transition from monotonicity to non-monotonicity when $\alpha=1$ and
% if $\alpha<1$ and that
the generic transition from continuous to non-continuous when
$\alpha=\gamma$.

%\pg{slight re-write here as I didn't understand ups and downs}

One attractive feature of the STS is that explicit expressions for
some of the bifurcations can be found.  A periodic solution on the
circle corresponds to a solution of the form $F^q(x)=x+p$ and has
rotation number $p/q$.
%modulo 1. In applications where the independent variable has a
%significance without taking modulo one (such as in the sleep models)
%the value of $p$ becomes important and w
We will refer to such solutions as $(p,q)$-periodic orbits, they
satisfy
%. Such solutions satisfy
\begin{equation}
x_{n+q} = x_{n} + p.
\label{eq:STMpqsolution}
\end{equation}
In the STS, a necessary condition for the existence of $(p,1)$-periodic solutions is that
there exist $x$ such that
\begin{equation}\label{eq:p1crit}
\sin 2 \pi x = \frac{2 \pi}{\alpha} \left ( p \tilde \gamma  - \beta \right )
-1, \quad\mbox{with} \quad\tilde \gamma = \frac{\gamma}{1+\gamma}.
\end{equation}
\begin{figure}
\centering{\includegraphics[scale=1.0]{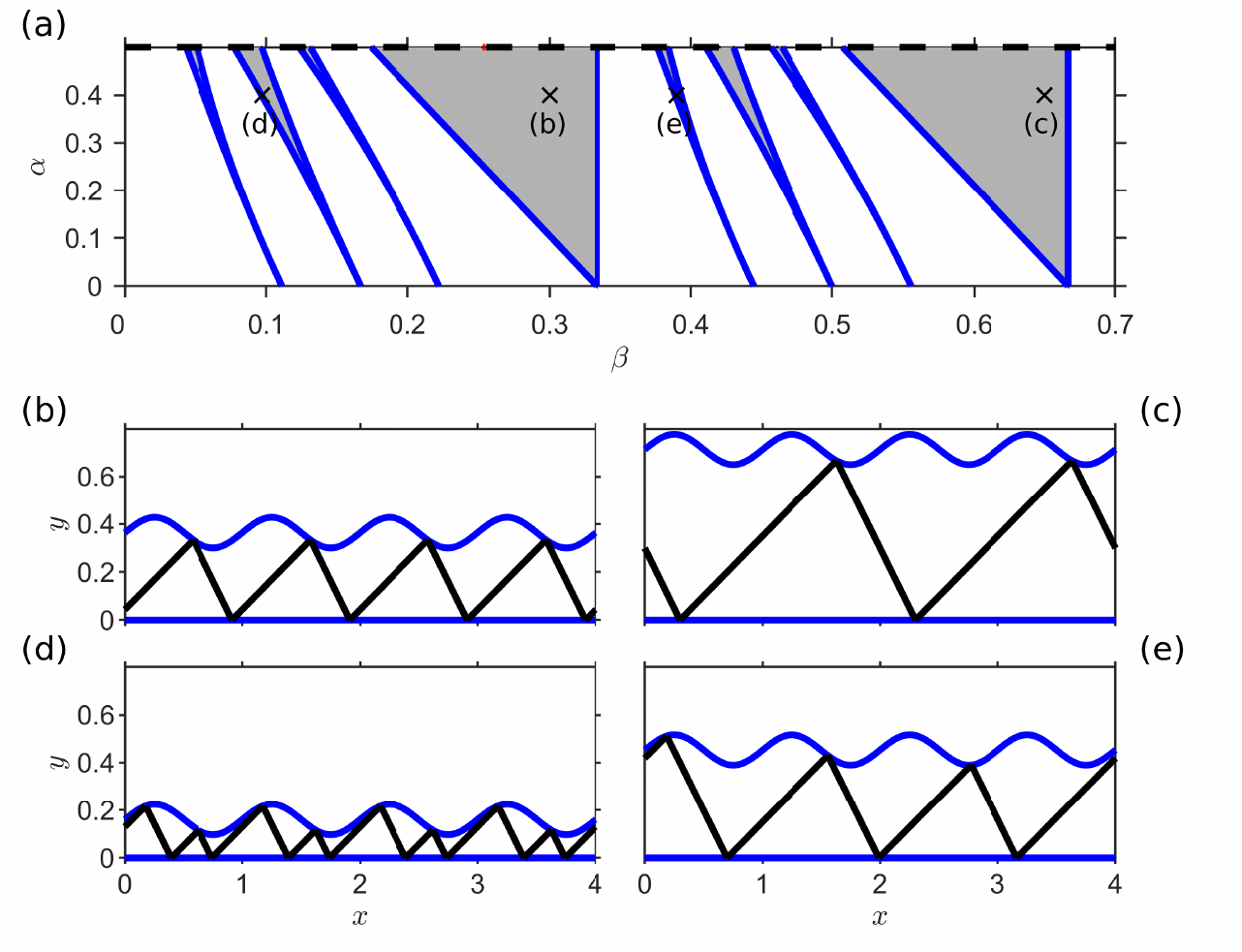}}

\caption{(a) Bifurcation set showing the largest few tongues that bound the
regions of existence for periodic solutions ($\gamma=0.5$). The blue lines are
lines of saddle-node bifurcations. (b)-(e) Trajectories for periodic
solutions with $(p,q)=(1,1),
(2,1), (1,2)$ and $(4,3)$ respectively ($\alpha=0.4, \gamma=0.5$ and
$\beta=0.3, 0.65, 0.097$ and $0.39$ respectively).
}
\label{fig:tongues_saddlenodes}
\end{figure}
As we will see in section~\ref{sect:tangencies}, this condition is not sufficient if the map has
a discontinuity. (This is related to the fact that a threshold system
is defined on the first intersection of the up/down flow with the
upper/lower threshold.) Solutions to~\ref{eq:p1crit} exist provided
its right hand side has modulus less than or equal to one. Thus for
fixed $\gamma$, the maximal region of existence of $(p,1)$-periodic
solutions in the $(\beta ,\alpha)$ parameter plane with $\alpha >0$ is
bounded by the curves
$$
\alpha  = \pi \left ( p \tilde \gamma - \beta \right ),
\quad\mbox{and}\quad
\beta=p\tilde \gamma,
$$
on which $x$ takes the values $1/4$ and $3/4$ respectively. If these
$x$ values correspond to first intersections of the up flow with the
upper threshold (as is the case when the map is continuous, i.e., for
$\alpha<\gamma)$, %Where they exist,
these curves are lines of saddle-node bifurcations that create one
stable and one unstable periodic solution. They form tongue-like
regions emanating from $\alpha=0$, $\beta=p\tilde \gamma$.

Saddle-node bifurcations for general $(p,q)$ tongues can be found by
numerically solving equation~(\ref{eq:STMpqsolution}) along with the
condition that the gradient of the $q^{\rm th}$ iterate of the map is
one.  Typical Arnold tongues for the STS with $\alpha<\gamma$ are
shown in Fig.~\ref{fig:tongues_saddlenodes}.
We note that there is a parameter symmetry for existence regions
for periodic solutions.  Specifically, if for some $\alpha, \beta, \gamma$
there exists a $(p,q)$-periodic orbit,
where $p,q \in \mathbb Z^+$, and $p$ and $q$ are relatively prime,
then there also exists a $(\tilde p,q)$-periodic orbit, where $\tilde p = p+mq$,
$m \in \mathbb Z^+$ for $\alpha, \gamma$ and
$$
\tilde \beta = \beta + \frac{m \gamma}{1+\gamma} =\beta + m\tilde\gamma .
$$
This symmetry is then reflected in the positions of the tongues: in
Fig.~\ref{fig:tongues_saddlenodes},
the $(2,1)$-tongue is a translation of the $(1,1)$-tongue,
the $(3,2)$-tongue is a translation of the $(1,2)$-tongue,
the $(4,3)$-tongue is a translation of the $(1,3)$-tongue and
the $(5,3)$-tongue is a translation of the $(2,3)$-tongue.

At $\alpha=\gamma$, the up flow becomes tangent to the upper
  threshold and the map looses smoothness at the pre-image of this
  tangency point (see~\eref{eq:mapgradient}). As we will show in the
  next section, the map develops a discontinuity for
  $\alpha>\gamma$. We will continue with this example after deriving
  the general theory.

%The fact that threshold systems can result in continuous circle maps
%that then lead to the presence of Arnold tongues is well-known. For
%example, in \cite{Glass_1991} {it is shown} %they show
%that in the limit that $\gamma \to \infty$, the STS becomes the
%standard circle map.  However, as we discuss in the next few sections,
%threshold systems also naturally result in circle maps with gaps which
%can be either monotonic or non-monotonic.

%%%%%%%%%%%%%%%%%%%%%%%%%%%%%%%%%%%%%%%%%%%%%%%%%%%%%%%%%%%%%%%%%%%%%%%%%%%%%%%%%%%%%%%%%
% Tangencies and gaps
%\input{Tangencies_gaps}

\section{Tangencies leading to gaps}\label{sect:tangencies}

As shown by Arnold~\cite{Arnold_1991}, gaps in a threshold map are a
result of `shadow' regions in the dynamics, that is, regions for which
the upper threshold is unreachable, as illustrated in
Fig.~\ref{fig:gaps} for the STS. In this example, there are regions of
the upper threshold such that every trajectory from the lower threshold
that intersects this region must already have crossed the upper
threshold at least once. For sufficiently smooth flows and thresholds,
the generic transition from no gaps to gaps will occur as a result of
either tangency of the up flow with the upper threshold or the down
flow with the lower threshold.
%
%In this section, we describe the properties of maps close to this
%transition, derive scalings for the size of the gap and discuss the
%consequences of gaps for the observed dynamics.

\begin{figure}
\begin{center}
\includegraphics[scale=1.0]{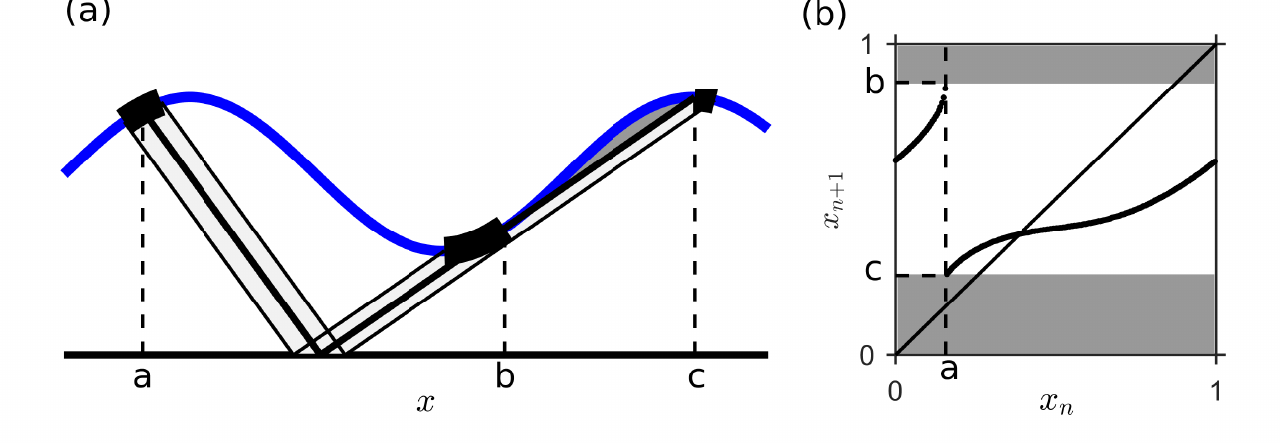}
\end{center}
\caption{(a) A tangency of the up-flow with the upper threshold for
  the STS ($\alpha=0.7, \beta=0.15, \gamma=0.5$). The tangency occurs at
  the point $x=b$ which has pre-image at $x=a$. The figure illustrates
  how the region local to $a$ maps to two disjoint sets, one local to
  $b$ and one local to $c$, where $x=c$ is the position of the second
  intersection.  The shadow region is shaded in dark grey and
  corresponds to $x \in (b,c]$. (b) Corresponding circle map.  }
\label{fig:gaps}
\end{figure}

%%%%%%%%% Start section 3.1
%\pg{imported my file -- please edit as you want}

\newcommand{\fun}{\mathcal{T}}
%\subsection{Existence of gaps}
In this section we will look at parameterised families of threshold
systems. We determine generic criteria for the creation of a
discontinuity and describe the local behaviour nearby. The
construction of the return map involves solving for the zeroes of a
function which can be treated in almost precisely the same way as the
standard bifurcation theory for fixed points (e.g. \cite{Kuz_1995},
chapters 4 and 8). Let $\mathbb{P}\subset\mathbb R$ be the parameter
space. Parameterised threshold maps can be thought of as the
composition of two maps: the down map
$T_d:\mathbb R\times \mathbb P\to \mathbb R$ from the upper boundary to the
lower boundary, and the up map $T_u:\mathbb R \times \mathbb P\to \mathbb R$
from the lower boundary to the upper boundary as described in
section~\ref{sect:threshold} but with the addition of a real
parameter, so $T_{u,d}=T_{u,d}(x,\mu)$. The periodicity of the
thresholds implies that both of these maps have period one in the
first variable: $T_{u,d}(x+1,\mu ) = T_{u,d}(x,\mu)$ for all
$x\in \mathbb R$.

Assume that the down map is a smooth bijection of the real line for all $\mu$ in the region of interest. Thus the down trajectory from $(x_n,h(x_n,\mu ))$ will intersect the lower threshold at $(x, g(x,\mu ))$ with $x=T_d(x_n,\mu )$ and for every $x\in \mathbb R$ such an $x_n$ exists.

The trajectory under the up flow $\phi$ starting at $(x,g(x, \mu ))$
on the lower threshold will intersect the upper threshold
$y=h(x, \mu)$, at time(s) $\tau$ which satisfy
\begin{equation}\label{eq:basic0}
W(\tau ,x,\mu ):=\phi_\tau (g(x,\mu), \mu )-h(x+\tau,\mu) =0.
\end{equation}
We will be interested in local behaviour near a solution
$(\tau^*,x^*,\mu^*)$ of (\ref{eq:basic0}) representing a first
intersection of the up flow with the upper boundary.
By shifting coordinates we may choose
$(x^*,\mu^*) = (0,0)$ and from now on we assume this shift of
coordinates has been implemented.
The essential genericity
condition on the $x$ and $\mu$ behaviour, assumed throughout this section, is that
\begin{equation}\label{eq:bgen}
W_2(\tau^*,0,0)\ne 0 \quad \textrm{and} \quad W_3(\tau^*,0,0)\ne 0.
%W_2(\tau^*,x^*,\mu^*)\ne 0 \quad \textrm{and} \quad W_3 (\tau^*,x^*,\mu^*)\ne 0.
\end{equation}
Here we used subscripts to denote
partial differentiation by the $i^{th}$ variable, e.g., $W_2 =
\frac{\partial}{\partial x} W$.

%\gd{Suggest to add figure Anne which describes the three cases
%  below. This figure can also be used to clarify the terminology
%  around the visible and invisible fold.}
%\acs{I've added some, but we might want to cut them down a bit}
%
\begin{figure}
\begin{center}
\includegraphics[scale=0.9]{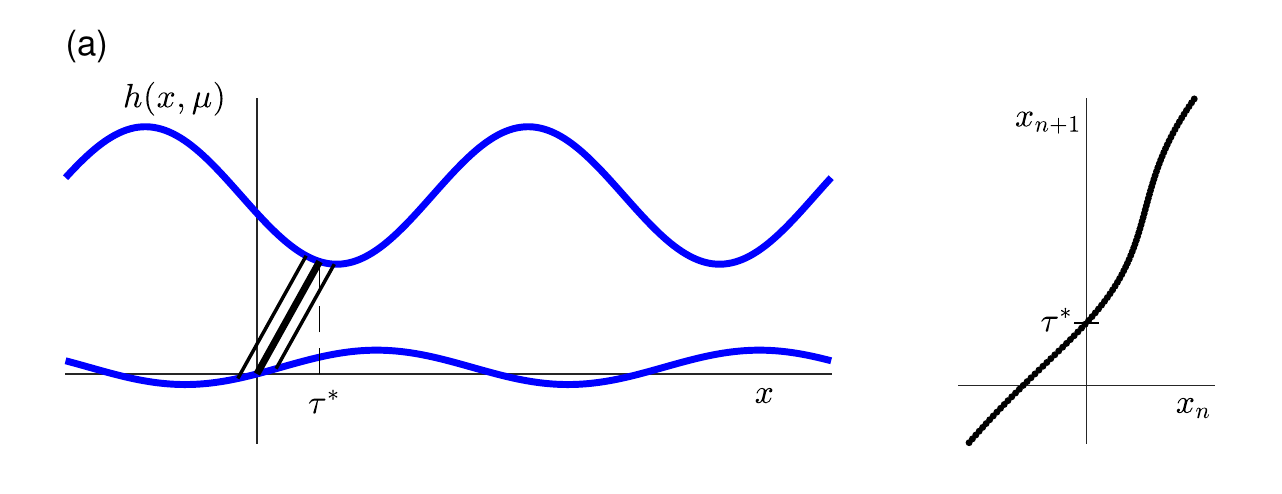}

\includegraphics[scale=0.9]{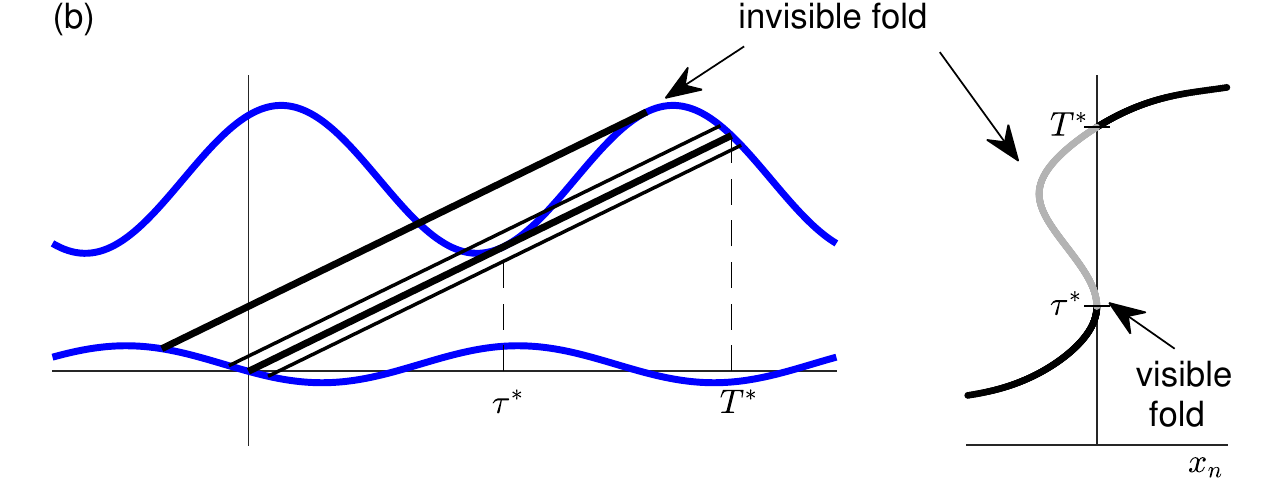}

\includegraphics[scale=0.9]{Figure4b.pdf}
\end{center}
\caption{Tangencies leading to gaps.
(a) Unique solution to $\tau^*$. (b) Existence of
a simple tangency between the up flow and the
upper boundary. (c) The cusp catastrophe.}
\label{fig:gapcreation}
\end{figure}

The map is locally continuous if in addition
%$W_1 (\tau^*,x^*,\mu^*)\ne 0$.
$W_1(\tau^*,0,0)\ne 0$. Indeed, a standard application of the
Implicit Function Theorem yields a unique and smooth local solution
$\tau (x,\mu)$ to~\ref{eq:basic0} the form
\[
%  \tau =\tau^*-(x-x^*)\,\frac{W_2}{W_1}+(\mu-\mu^*)\,\frac{W_3 }{W_1}+ O(2)
  \tau =\tau^*-x\,\frac{W_2}{W_1}+\mu\,\frac{W_3 }{W_1}+ O(2)
\]
where the partial derivatives are evaluated at
%$(\tau^*,x^*,\mu^*)$.
$(\tau^*,0,0)$ (see Fig.~\ref{fig:gapcreation}(a)).

There is a simple tangency between the upwards trajectory and the upper boundary if
\begin{equation}\label{eq:tan}
W(\tau^*,0,0)=0, \quad W_1(\tau^*,0,0)=0, \quad W_{11}(\tau^*,0,0)\ne 0.
\end{equation}
Although this is not
%the bifurcation that creates
the \emph{first}
tangency it is worth considering as it shows the persistence of jumps
in the one-dimensional map.
In this case, the intersection equation~(\ref{eq:basic0}) can be written locally as
%there exists a unique local solution~to (\ref{eq:basic0}) in the form
%of a fold:
\begin{equation}\label{eq:x2}
(\tau -\tau^*)^2=-\frac{2}{W_{11}}\left( W_2x+W_3\mu \right) + O(2).
\end{equation}
Hence there is a fold along a curve in the $(x,\mu )$-plane given by
$W_2x+W_3\mu =0$ in lowest order. The fold has two interpretations: it
represents a unique local solution to~\eref{eq:basic0} (a fold point)
$\tau(\mu) = \tau^* + O(\mu)$. In the threshold system it gives also a
persisting simple tangency between the upwards trajectory and the
upper threshold at this point. When $W_{11}(W_2x+W_3\mu)< 0$
then~\eref{eq:basic0} has locally two solutions, one less and one
greater than $\tau^*$. The map is defined using the negative solution
to (\ref{eq:x2}) as this represents the first intersection of the up
flow and the upper threshold. Locally the second solution does not play
a role in defining the map.  Furthermore, the derivative
$\frac{\partial \tau}{\partial x}=\frac{W_2}{W_{11}(\tau^*-\tau)}$ in
lowest order, hence for $(x,\mu)$, with $W_{11}(W_2x+W_3\mu)<0$,
approaching the fold curve, the derivative of the map becomes
unbounded at $\tau^*$ and exhibits a square root singularity
(see Fig.~\ref{fig:gapcreation}(b)).
%There are two possible interpretations of this result, which at the
%risk of confusing with the standard names in piecewise smooth
%dynamics (of which this is a natural counterpart) we refer to as a
%visible fold and an invisible fold.

By Definition~\ref{def:TS} %assumption
there are solutions to $W(\tau ,x,\mu )=0$ even if
$W_{11}(W_2x+W_3\mu)>0$.
% the right hand side of (\ref{eq:x2}) is negative.
Generically this implies that~\eref{eq:basic0} has a second
(typically non-local) solution $(T^*,0,0)$, $T^*>\tau^*$, representing
the second intersection between the up flow and the upper threshold
(after the first one at $(\tau^*,0,0)$). Generically all the first
derivatives at $(T^*,0,0)$ are non-zero and the map can be continued
for $W_{11}(W_2x+W_3\mu)>0$, with a discontinuity and finite
derivative on approaching the fold curve
$W_2x+W_3\mu=O(2)$ from this side.
% If $\tau -\tau^*<0$ and is in the physically
% meaningful region then this means that the upwards flow intersects the
% surface $h(x,\mu )=0$ for the first time at a regular point and the
% fold above it is immaterial, hence we call the fold \emph{invisible},
% and the local solution is not part of the return map.
%
% If, on the other hand, $\tau -\tau^*>0$ then (generically) if
% $W_{11}(W_2x+W_3\mu)\le 0$ then the map is defined using the lower
% solution to (\ref{eq:x2}), whilst if  $W_{11}(W_2x+W_3\mu)>0$ the
% solution jumps to the upper branch near $T^*$.
This proves that once the jump exists, it is stable to perturbations
of the map.
This analysis also implies that there is a second fold between
$\tau^*$ and $T^*$ for some fold curve $(x,\mu)$ with
$W_{11}(W_2x+W_3\mu)<0$. This fold is not part of the return map and
at the risk of confusing with the standard names in piecewise smooth
dynamics (of which this is a natural counterpart) we refer to this
second fold as an invisible fold and to the fold at $\tau^*$ as a
\emph{visible} fold,  %. \gd{Add reference to figure Anne}
see Fig.~\ref{fig:gapcreation}(b).

So at a standard fold satisfying (\ref{eq:tan}) the discontinuity
already exists ($T^*$ of the previous paragraphs is non-local). To
obtain the transition from continuity to discontinuity we need a
further condition:
\begin{equation}\label{eq:tantan}
W(\tau^*,0,0)=W_1(\tau^*,0,0)=W_{11}(\tau^*,0,0)= 0, \quad W_{111}(\tau^*,0,0)\ne 0,
\end{equation}
see Fig.~\ref{fig:gapcreation}(c).
Taken together with the genericity condition (\ref{eq:bgen}) this
defines a standard unfolding of the cubic singularity: a cusp
catastrophe. The standard form for the unfolding of the cusp is
\begin{equation}\label{eq:cuspa}
A +B u+u^3=0, \quad u=\tau -\tau^*,
\end{equation}
 which has three solutions if $A^2<\frac{4}{27}|B|^3$, $B <0$ and
 otherwise one solution except on the curves $A^2=\frac{4}{27}|B|^3$,
 $B <0$, which are the loci of local folds. Using the Taylor series
 for $W$, the intersection equation (\ref{eq:basic0}) becomes
\[\begin{array}{rl}
0= & (W_2x+W_3\mu + O((|x|+|\mu |)^2) )+(W_{12}x+W_{13}\mu +O((|x|+|\mu |)^2 )u\\
& +\frac{1}{2}(W_{112}x+W_{113} \mu +O((|x|+|\mu |)^2 )u^2+\frac{1}{6}(W_{111}+O(|x|+|\mu | )u^3+O(u^4).
\end{array}
\]
The transformation $u\to u-\frac{1}{W_{111}}( W_{112}x+W_{113} \mu)$ transforms this to
\begin{equation}\label{eq:cuspb}
0=a(x,\mu )+b(x,\mu )u+u^3+O(u^4),
\end{equation}
where
\begin{equation}\label{eq:axmu}
%  a(x,\mu )=W_2x+W_3\mu +O((|x|+|\mu |)^2)
    a(x,\mu )=\frac{6}{W_{111}}\left(W_2x+W_3\mu\right) +O((|x|+|\mu |)^2)
\end{equation}
and
\begin{equation}\label{eq:bxmu}
%b(x,\mu )=(W_{12}-\frac{1}{6}W_{112})x+(W_{13}-\frac{1}{6}W_{113})\mu
%+O((|x|+|\mu |)^2).
  b(x,\mu )=\frac{6}{W_{111}}\left(W_{12}x+W_{13}\mu\right) +O((|x|+|\mu |)^2).
\end{equation}
Note that $a(0,0)=b(0,0)=0$. Provided the transformation from $(a,b)$ to $(x,\mu )$ is non-singular, i.e. provided
\begin{equation}\label{eq:Jac0}
\det\left(\frac{\partial (a,b)}{\partial (x,\mu)}\right)\Big|_{(0,0)}\ne 0,
\end{equation}
all the standard results for the unfolding of the $x^3$ singularity in the $(a,b)$-plane carry over to the $(x,\mu )$-plane. A glance at the standard cusp manifold shows that one of the branches of the folds defined by the cusp is invisible and the other is visible, creating a bifurcation curve of jumps in the $(x,\mu )$-plane terminating at $(0,0)$.

\begin{lemma}\label{thm:cusp}Let $a(x,\mu )$, $b(x,\mu )$ be as
  defined in~\eref{eq:axmu} and~\eref{eq:bxmu}. Suppose that the
  nondegeneracy condition (\ref{eq:bgen}) and the double tangency
  condition (\ref{eq:tantan}) hold with
\begin{equation}\label{eq:ngen}
  % W_2(W_{13}-\frac{1}{6}W_{113})-W_3 (W_{12}-\frac{1}{6}W_{112})\ne 0.
  W_2W_{13}-W_3 W_{12}\ne 0.
\end{equation}
Then for $\sigma\in\{+1,-1\}$ the locus $a=\sigma
\frac{2}{3\sqrt{3}}|b|^{\frac{3}{2}}$, $b<0$, gives a fold at
$u = \sigma \sqrt\frac{|b|}{3}$.
\end{lemma}
\begin{proof}
%\emph{Proof:}
  Equation (\ref{eq:ngen}) is just a rewriting of (\ref{eq:Jac0}). The
  remainder is a restatement of properties of the standard
  unfolding of the singularity $\pm u^3$.
\end{proof}
%\newline\rightline{$\square$}

Lemma~\ref{thm:cusp} is a re-interpretation of the cusp bifurcation
(\cite[section 8.2]{Kuz_1995}). The up map $T_u$ is
\[
x \to x+u_m+\tau^*-\frac{W_{112}x+W_{113} \mu}{W_{111}}
\]
where $u=u_m$ is the smallest $u$ value satisfying (\ref{eq:cuspb}),
hence the visible fold corresponds to $\sigma=-1$ in
Lemma~\ref{thm:cusp}.  The lemma determines the locus of the
discontinuity but not the effect.

\begin{corollary}\label{cor:jumps}Suppose that (\ref{eq:bgen}),
  (\ref{eq:tantan}) and~\eref{eq:ngen} hold.
  % Suppose the conditions of Theorem~\ref{thm:cusp} hold.
  Then the corresponding up map $T_u$ has a discontinuity of size of
  order $\sqrt{|\mu |}$ on the visible fold line.
  % and on one side of
%   the discontinuity the map has slope which tends to infinity like
%   $|x|^{-\frac{1}{2}}$.
\end{corollary}
\begin{proof}
  % \emph{Proof:}
The size of the discontinuity is the change in $\tau$ values from the
value on the visible fold to the other value of the cubic
(\ref{eq:cuspb}). On the visible fold  $u^2=-\frac{1}{3}b>0$ to lowest
order, and if we choose the negative solution,
$u=-\sqrt{\frac{-b}{3}}$ then substituting back into (\ref{eq:cuspb})
gives $a=-\frac{2}{3\sqrt{3}}|b|^{\frac{3}{2}}$ \cite{Kuz_1995}. At
the fold there are only two distinct solutions. The solution
$u\approx -\sqrt{\frac{-b}{3}}$ already discussed is a repeated root,
and so the other solution can be found by solving
$(u+\sqrt{\frac{-b}{3}})^2(u-\alpha )=a+bu+u^3$ to give $\alpha =
2\sqrt{\frac{-b}{3}}$ to lowest order. The jump is therefore the
difference of the two roots, i.e. $3\sqrt{\frac{-b}{3}}$.

On the fold $a^2=-\frac{4}{27}b^3$, with $a$ and $b$ linear functions
of the original variables $x$ and $\mu$ via (\ref{eq:axmu}) and
(\ref{eq:bxmu}). To lowest order this implies that the fold is
tangential to the line $a=0$, i.e. $x\sim -W_3\mu/W_2$ and
so the evaluating $b$ on this line the jump becomes
\[
3\sqrt{2\Big| \frac{W_{12}W_3-W_{13}W_2}{W_2W_{111}}\mu\Big|}
\]
%where $B_k$ are the coefficients of $x$ and $\mu$ respectively in
%(\ref{eq:bxmu}).
Thus the jump is order  $\sqrt{|\mu |}$. %as stated.
% The map is $x\to x +\tau$ and so the derivative is
% $1+\frac{d\tau}{dx}$. Since
% $\tau =u +\tau^*-\frac{W_{112}x+W_{113} \mu}{W_{111}}$ with $u$ is a
% solution of (\ref{eq:cuspb}), implicit differentiation of
% (\ref{eq:cuspb}) gives
% \begin{equation}\label{eq:uder}
% \frac{d\tau}{dx}=\frac{du}{dx}-\frac{W_{112}}{W_{111}}=-\frac{a^\prime
%   +b^\prime u + O(u^4)}{b+3u^2+O(u^3)}-\frac{W_{112}}{W_{111}}.
% \end{equation}
% The denominator of (\ref{eq:uder}) is zero precisely on the visible fold, and the nature of the singularity is clear if we expand $u=-\sqrt{\frac{-b}{3}}+v$, $v\ll 1$ near the singularity:
% \[
% \frac{dv}{dx} =\frac{1}{Cv}
% \]
% to lowest order, where $C$ is a constant, and so $Cv^2\sim (x-x_c)$,
% where $x_c$ is the position of the fold. This leads to the square root
% singularity in $x$ as stated.
\end{proof}
%\newline\rightline{$\square$}

Note that although we have derived these properties for the up map alone, they are preserved by composition with a smoothly invertible down map.

%\pg{Is this enough detail? Suggest next section replaced entirely by
%  a numerical verification of predictions.}
%\gd{Not sure how much such numerical section adds. If we have space, I would
%  suggest to use it to summarise the results in this section instead.}
%\acs{Let's discuss}
%
%%%%%%%%% End section 3.1
%\subsection{Numerical Verification}
%In this subsection we will provide numerical evidence for the square root growth of the discontinuity predicted by
%Corollary~\ref{cor:jumps} for the STS of section~\ref{sect:STS}. There can be no tangency between trajectories of the up flow and the upper threshold if $\gamma >\alpha$ since the maximum slope of the upper threshold is $\alpha$ (at $x\in {\mathbb N}$). If $\gamma =\alpha$ then there is a tangency, and hence it is the first tangency, between the trajectory from $y=0$ with $x_0=1-b-\frac{1}{2\pi}\alpha$ to the upper threshold at $x=1$.
%
%We have computed the size of the resulting jump when $\gamma <\alpha$ as a function of $\gamma$ with $\alpha$ and $\beta$ fixed. The results are shown in the log-log plot of Figure~??, which is close to a straight line with gradient $\frac{1}{2}$ as expected.
%
%
%
%%%%%%%%%%%%%%%%%%%%%%%%%%%%%%%%%%%%%%%%%%%%%%%%%%%%%%%%%%%%%%%%%%%%%%%%%%%%%%%%%%%%%%%%%
% Square root map
\section{Square root discontinuities in monotonic circle
  maps}\label{sect:squareroot}

The discontinuous circle maps derived in section~\ref{sect:tangencies}
have two features: there is an interval of values which cannot be
reached by the map, and on one side of this gap the derivative of the
map tends to infinity. In this section, we will consider the
fundamental bifurcation sequences for periodic solutions of the
simplest form of such maps and illustrate the observations with two
examples.

In terms of the lift $F$ of such a  circle map, defined on the
real line with $F(x+1)=F(x)+1$, coordinates can be chosen such that
\[
\lim_{x\downarrow 0}F(x)+1>\lim_{x\uparrow 1}F(x) ,
\]
and $F$ is continuously differentiable and strictly increasing on $(0,1)$. In other words
the gap is a jump upwards at integer values of $x$. The square root discontinuity
implies that
\[
 \lim_{x\downarrow 0}F^\prime (x)=c>0,\quad  F^\prime (x)\to \infty~\textrm{as}~x\uparrow 1,
\]
or vice versa. Thus $F$ is strictly increasing and hence has a unique rotation number \cite{Rhodes_1986}.

Now consider families $(F_\mu )$ of such maps. Provided $F_\mu$ is
continuous in $\mu$ then the rotation number is a continuous function
of $\mu$ (\cite[Theorem 5.8]{Rhodes_1991}). If $F_\mu$ is an
increasing function of $\mu$ for each fixed $x$, then the rotation number
is monotonic in $\mu$, irrational values of the rotation number
exist at isolated values of $\mu$ (\cite[Proposition
6.1]{Rhodes_1991}), and the invariant set for an irrational
  rotation number is a Cantor set
\cite{Keener_1980}.  Finally, under the same conditions, the set of
$\mu$ with a given rational rotation number is a non-trivial closed
interval, i.e. the maps have phase locking analogous to that of
continuous homeomorphisms of the circle (\cite[Theorem
6.6]{Rhodes_1991}).

The gaps introduce a second mechanism by which phase-locked solutions
can be created or destroyed, that is via border collisions.
We will call a border collision with an end point with an infinite
derivative a \emph{type I} border collision and use \emph{type II} for
a border collision with an endpoint with a finite derivative.
%
%\gd{\\This is consistent with what we used in the EJAM paper, but
%  sounds a bit weird stated like this.\\}%
%
Within each phase-locked region, there is a branch of periodic
solutions connecting $0^+$ with $1^-$. If $F_\mu^\prime (x)<1$ for all
$x\in (0,1)$ then there are two border collision bifurcations, one
creating a stable fixed point and the other destroying it, and
similarly for $F_\mu^q$ in a $p/q$ phase locked region. There can be no
saddle-node bifurcations in such maps.
However, for maps with a square root singularity $F_\mu^\prime (x)$
cannot be less than one for all $x\in (0,1)$, although to remain an
injection $F_\mu^\prime (x) <1$ for some $x\in (0,1)$. For maps
$F_\mu(x)$ that are monotonic increasing in both $x$ and $\mu$, the
border collision with the square root singularity at $1^-$ always
leads to the creation of an unstable periodic orbit when $\mu$ is
increased. The border collision with~$0^+$ can either lead to the
creation of a stable periodic orbit (if $(F^q)^\prime (0^+)<1$) or the
loss of a unstable periodic orbit (if $(F^q)^\prime (0^+)>1$).

%Hence there is always a stable periodic orbit in each phase locked
%region
%
%\gd{\\We need some argument to explain this statement. Why does the
%  stable one come in for each higher iterate and can we rule out the
%  BC$\to$BC case for $q>1$? Note that for our two examples we can use
%  a continuity argument to rule out this case. In the square root map,
%  if $c=0$, $f^q(0)=0$ for all $q$. And $f$ is convex, hence case (a)
%  happens for $c=0$. Because of the continuity in $c$, the saddle-node
%  curve can not be destroyed by the infinite derivative border
%  collision.\\}%
%%
%and
There appear to be two possible `simplest' sequences of bifurcations for the
creation/destruction of periodic orbits of a given rotation number
$p/q$:
\begin{enumerate}
\item[(a)] border collision $\to$ border collision $\to$ saddle-node bifurcation;
\item[(b)] saddle-node bifurcation $\to$ border collision $\to$ border
  collision $\to$ saddle-node bifurcation.
\end{enumerate}
These two bifurcation sequences are illustrated in the bifurcation
diagram shown in Fig.~\ref{fig:bifseq}.  Case (a) occurs if
$(F^q)^\prime (0^+)<1$ when first a stable periodic orbit is created
in a border collision at $0^+$, then a second unstable periodic orbit
is created in a border collision at $1^-$ and finally both solutions
disappear in a saddle-node bifurcation.  Case~(b) occurs if
$(F^q)^\prime (0^+)>1$.  In both bifurcation sequences, there is
  always a stable periodic solutions in each phase-locked region.  Of
course, %as noted earlier,
the bifurcation sequences could have %re may be
extra pairs of saddle-node bifurcations. However, if $F_\mu$ is
convex, then also $F^q_\mu$ is piecewise convex and case (a) will
always occur as the convexity implies that $(F^q)^\prime$ is monotonic
increasing, hence at most one saddle-node bifurcation is possible.
%and/or border collision bifurcations whose net effect is to add no changes to the
%stability of orbits, or create isolas.
Note also that $F^\prime (0^+)\lessgtr1$ does not imply that
$ (F^q)^\prime (0^+)\lessgtr1$ for all $q>1$, so it may be possible to find
both case (a) and case (b) in examples.
% However, $F(1^-)-F(0^+)<1$ implies
% that $F^\prime (x)<1$ for some $x$ which implies that there are stable
% periodic solutions within each phase locked region.
%
%Note that within each phase locked region there is always a stable periodic solution.
%

\begin{figure}
\begin{center}
\includegraphics[scale=1]{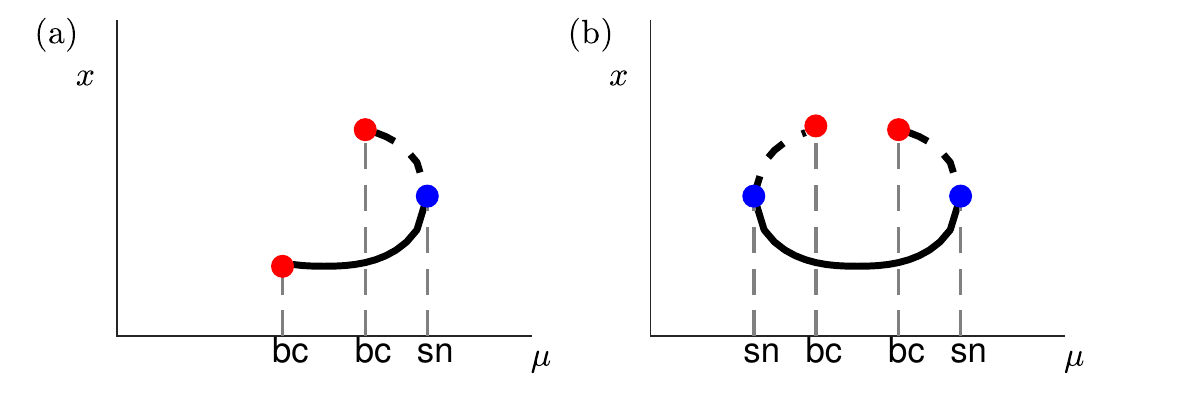}
\end{center}
\caption{Bifurcation diagrams illustrating the two sequences (a)
border collision $\to$ border collision $\to$ saddle-node bifurcation;
(b) saddle-node bifurcation $\to$ border collision $\to$ border
  collision $\to$ saddle-node bifurcation.
}
\label{fig:bifseq}
\end{figure}

The monotonicity of the map in the parameter ensures that structures
are not repeated; if this parameter monotonicity assumption is
not true, the essential results hold true but there may be multiple
parameter values with a given irrational rotation number (or even
intervals), and the bifurcation structure of phase-locked regions can
be considerably more complicated. The important observation is that
the number of times $0^+$ (resp. $1^-$) is periodic in a phase locked
region must be odd (counting with multiplicity, so a solution that
approaches 0 from above and then moves back up again is counted as two
intersections) if the rotation number changes from below a given
rational number to above it, and so at least one branch of solutions
connects $0^+$ to $1^-$.  For the fundamental cases (a) and (b) this
implies that there may be extra pairs of saddle-node bifurcations
and/or border collision bifurcations whose net effect is to add no
changes to the stability of orbits, or create isolas.

In the next sections, we illustrate the bifurcation sequences (a) and
(b) with a `canonical' example and then show how they are manifested
in the STS.

\subsection{Example: `canonical' square root singularity maps}
To illustrate the consequences of the square root singularity we have
constructed an example which is monotonic in $x$ and its parameters
and demonstrates the bifurcation sequences described above. This
example can be derived from threshold models, by using the inverse to
define an upper threshold with the gap replaced by %filled in with
a smooth continuation and taking %having
the lower threshold to be a constant. %, although we will not give the construction here.
% \pg{Is this obvious? Should we say something?} \gd{Added a line, is
%   this sufficient?}

% \gd{I have merged to two examples, keeping the explicit calculations for
%   $n=2$.}
% The first example makes it possible to describe the structure of the
% model-locked region corresponding to rotation number zero (i.e. the
% fixed points) analytically. This gives the first possibility (a):
% saddle node bifurcation with tow border collision bifurcations, one
% stable and the other unstable. The second is more general and shows
% that the second possibility (b) can also occur.

For $n\geq 2$, define the lift $F_n$ by
\begin{equation}
  F_n(x)  =
  a +b\left(1 - c_1 \sqrt{1-x} +
    (c_1-1)\,\left(\sqrt{1-x}\, \right)^n\right),
  \quad x\in [0,1), \label{eq:sqrmap}
\end{equation}
and extend $F_n$ to the real line using the periodicity condition
$F_n(x+m)=F_n(x)+m$, $m\in{\mathbb Z}$.
Taking
\[
  c_1 = \frac{nb-2c}{(n-1)b},
\]
gives
\begin{equation}\label{eq:mon_incr}
  F_n(0)=a, \quad F_n^\prime (0)=c, \quad \textrm{and}\quad
  \lim_{x\uparrow 1}F_n(x)=a+b.
\end{equation}
The assumption $b\in(0,1)$ gives a discontinuity at
$x=m\in\mathbb{Z}$. Since
\begin{equation}\label{eq:sqrtderiv}
F_n^\prime (x) = \frac{b}{2\sqrt{1-x}} \, \left( c_1+(1-c_1)n\left(\sqrt
    {1-x}\, \right)^{n-1} \right),
\end{equation}
$F_n$ is the lift of a monotonic circle map if
$0<c_1\leq\frac{n}{n-1}$, i.e. $0\le c<\frac{nb}{2}$.  Note that at
$c=\frac{nb}{2}$, $c_1=0$ and the square root singularity disappears.
The lift $F_n(x)$ is increasing in the parameter~$a$.
%Increasing $c$ further would create turning points so the map is no
%longer a monotonic function.

First we look at the case $n=2$, see
Fig.~\ref{fig.mapplots}(a),
as the structure of the mode-locked
region corresponding to rotation number zero (i.e. the fixed points)
can be described analytically. We have
\begin{equation}\label{eq:simplesqr}
  F_2(x) %= a+ b + (b-2c)(1-x) -2(b-c)\sqrt{1-x}
  = a+ 2(b-c) + (2c-b)x +2(c-b)\sqrt{1-x}.
  %F(x)  = a  +b x -c\sqrt{1-x},
  %   \quad x\in [0,1),
\end{equation}
Since
\[
  F_2^\prime (x)=2c-b+\frac{b-c}{\sqrt{1-x}}, \quad x\in(0,1),
  %b+\frac{c}{2\sqrt{1-x}},
\]
$F_2$  is a monotonic increasing function of the real line with a
quadratic singularity at $x=m\in \mathbb{Z}$ provided
$0\leq c<b$. This implies that $F_2$ is convex and $F_2'(0)=c<b<1$.
%\[
%  c<b \quad \textrm{and}\quad c\geq 0
%c>0 \quad \textrm{and}\quad -\textstyle{\frac{1}{2}}c\le b\le 1-c,
%\]
%(the final inequality is from the condition $F(0)+1\ge
%\lim_{x\uparrow 1}F(x)$).
Fixed points of $F_2$ in $(0,1)$ satisfy
\[
x=a+b + (b-2c)(1-x) -2(b-c)\sqrt{1-x}.
\]
and these are created in either border collisions, in which the fixed
point is at the discontinuity (0 or 1) or saddle-node bifurcations (as
the map is increasing there can be no other smooth bifurcations).

Thus there is a border collision at $x=0$ if $a=0$ and since
$F_2'(0)<1$,  this is a border collision creating a stable fixed
point in $a>0$ locally. There is a border collision at $x=1$ if
$a=1-b$, and this is an border collision creating an unstable
fixed point in $a>1-b$ locally.
Saddle-node bifurcations occur if $F_2^\prime (x)=1$ at a fixed point.
% , i.e. a fixed point with
% \[
% 2c-b+\frac{b-c}{\sqrt{1-x}}=1. %b+\frac{c}{2\sqrt{1-x}}=1.
% \]
A straightforward manipulation shows that the locus of these bifurcations is
\begin{equation}\label{eq:locsnb}
  % a=1-b+\frac{c^2}{4(1-b)}.
  a=1-b+\frac{(b-c)^2}{1+b-2c}.
\end{equation}
Fig.~\ref{fig.mapplots}(b) shows the bifurcation structure in the
$(a,c)$-plane for $b=0.7$. %\frac{1}{2}$.
As $a$ increases there is a border collision bifurcation at
$a=0$ which creates %/destroys
a stable fixed point, followed by a border collision bifurcation at
$a=0.3$ %\frac12$
which creates %/destroys
an unstable fixed point and finally the saddle-node bifurcation
destroying the pair of orbits created in the two border collision
bifurcations. This is precisely case~(a) as described at the start of
this section.%~\ref{sect:squareroot}.

\begin{figure}[htb]
\centering
\includegraphics[scale=1]{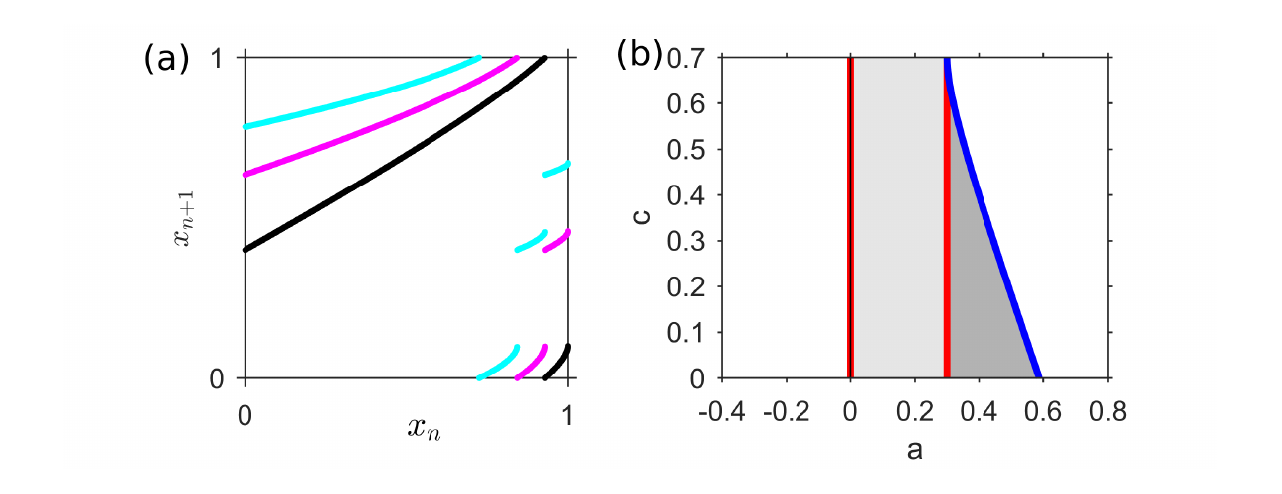}
\caption{
(a) The map with lift $F_2(x)$ (black) and its second (magenta) and third
(cyan) iterates for $a=0.4$, $b=0.7$, $c=0.5$.
(b) The bifurcations in the $(a,c)$-plane of the fixed points
($\rho=0$) of the map with lift $F_2$ defined by
(\ref{eq:simplesqr}) with $b=0.7$. Border collision bifurcations
are shown in red and saddle-node bifurcations are shown in blue. In the
light/dark grey regions there are one/two fixed points.}
\label{fig.mapplots}
\end{figure}

Next we return to the general map $F_n$.
%Two examples of the
%associated circle maps are shown in
%Fig.~\ref{fig.mapplots}(b-c). \gd{Note sure how useful these
%  are. Suggest to skip middle plot and do last one for $n=2$.}
  Since
the map has a discontinuity at $x=1$, if the range of $F_n^{m-1}$
contains an integer, then the $m^{th}$ iterate of the circle map $f_n$
with lift~$F_n$ has $m$ discontinuities at the pre-images of
$x=1, f_n^{-1}(1), \ldots , f_n^{-(m-1)}(1)$, and the derivative of
$f_n^m$ is unbounded on one side of each of these pre-images as
illustrated in Fig.~\ref{fig.mapplots}(a).  If the map is convex, then
the higher iterates are piecewise convex.

\subsubsection{Periodic solutions}

For $a\in[0,1]$, $b\in(0,1)$ and $c\in \left[0,\frac {nb}2\right)$,
the lift $F_n(x;a,b,c)$ is monotonically increasing in $x$ (hence
orientation preserving) and piecewise continuous so the rotation
number is well-defined and if it is rational it can be linked to
periodic solutions of the map. The map is also increasing in all its
parameter as the derivatives of the lift $F_n$ with respect to the
parameters are positive. Indeed %\gd{can skip this detail?}
\[
\frac{\partial}{\partial a} {\Map}(x) = 1; \quad
%\frac{\partial}{\partial b} {\Map}(x) = 1 - \frac{\sqrt{1-x}}n-1}
%\left(p- \left(\sqrt{1-x}\right)^{p-1}\right)
\frac{\partial}{\partial b} {\Map}(x)
=\frac{(1-\sqrt{1-x})}{n-1}\,
\left(n-\sum_{j=0}^{n-1}\left(\sqrt{1-x}\right)^j\right);
\]
\[%\quad
\frac{\partial}{\partial c} {\Map}(x) =
\frac{2\sqrt{1-x}}{n-1}\left(1-\left(\sqrt{1-x}\right)^{n-1}\right);
\quad x\in[0,1].
\]
% As the map depends on the parameters $a$, $b$, and $c$,
% the rotation number will depend on those parameters:
As the maps are increasing injections,  the rotation number
\[
\rho(a,b,c) = \lim_{k\to\infty}\frac{F_n^k(x;a,b,c)-x}{k}.
\]
is well-defined and increases continuously~\cite{Rhodes_1986}
from~$\rho=0$ to~$\rho=1$ when~$a$ increases from~$0$ to~$1$ and~$b$
and~$c$ are fixed. Since $F(0^+)=a$, it follows immediately that if
$a=m\in\mathbb{Z}$, then $\rho(m,b,c)=m$ for any $b,c$.
%-- we assume these conditions hold for the remainder of this section
%along with $0< c <\frac{bn}{2(n-1)}$, a condition which implies that
%$F_n$ is convex, i.e. $F^{\prime\prime}(x)>0$ if $x\in (0,1)$.

Consider the border collision bifurcations of fixed points for $F_n$ in $(0,1)$, i.e.,
the $a$ values for which $\rho\in\mathbb{Z}$. There is a border collision
bifurcation if either $F(0)=0$ or $\lim_{x\uparrow 1}F(x)=1$. These
occur at $a=0$ and $a=1-b$ respectively.  As $a$ increases, the latter
always destroys an unstable fixed point, whilst the former creates a
stable fixed point if $c<1$ and an unstable fixed point if
$c>1$. Compared to $F_2$, the new case here is that if $c>1$ then both
border collision bifurcations can be unstable. In this case the
simplest bifurcation sequence is to have case~(b): two saddle-node bifurcations
bounding the mode-locked region.
% Note that in this case $F_n$
% cannot be convex as the derivative must be less than one in some
% regions of phase space and so $F^\prime_n$ cannot be increasing.
Further bifurcation curves have to be determined numerically.
\begin{figure}[htb]
%  \centering
%  \mbox{\hspace*{-1cm}\includegraphics[width=.75\textwidth]{Figures/fourth_iterate_b9_p5_v3}\hspace*{-1cm}
%\includegraphics[width=.4\textwidth]{Figures/BC_cross_SN}}
\begin{center}
\includegraphics[scale=1]{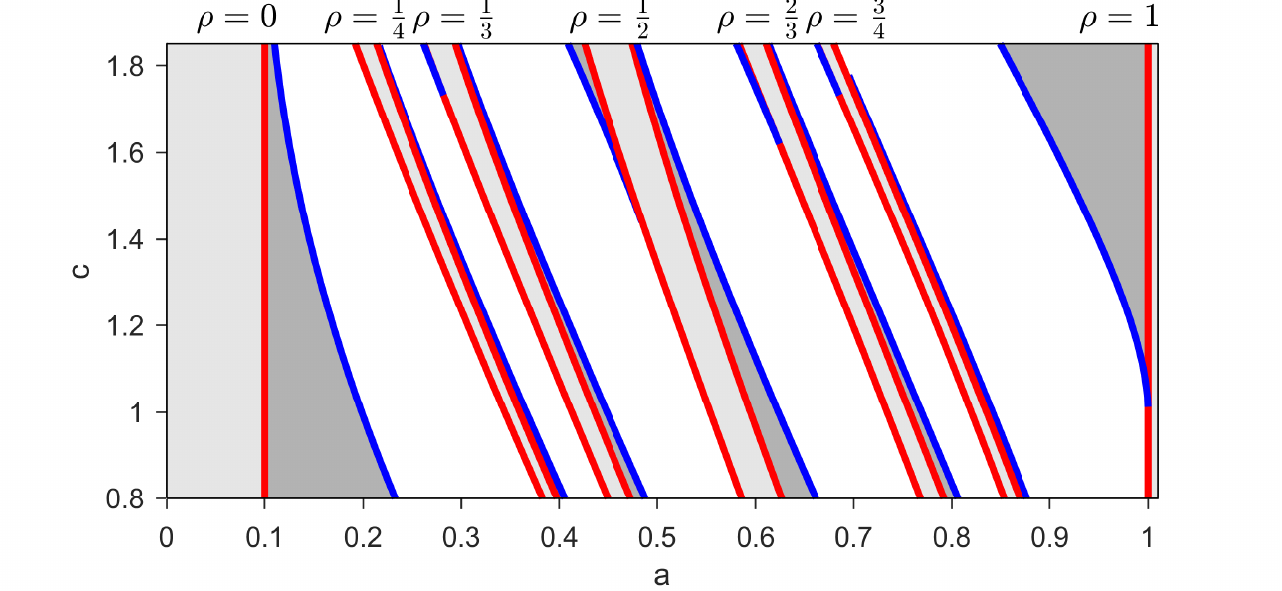}
\end{center}
\caption{The bifurcation set in the $a$-$c$ plane
for the map $F_n(x)$ up to the fourth iterate. Parameters:
$b=0.9,$ $n=5$, $c\in[0.8,1.85]$.
  The light/dark shaded regions are regions where
  one/two fixed points exist with the labelled rotation number.
  Transitions between different numbers of fixed points are either
  saddle-node bifurcation curves (blue) or  border collisions (red).
  For each pair of border collision
  curves, the right curve is the type I border collision
  % ($a_{\rm BC_{I}}$)
  and the left curve is the type II border collision.
  % ($a_{\rm BC_{II}}$).
}
\label{fig.bif5_9}
\end{figure}

In Fig.~\ref{fig.bif5_9} the saddle-node bifurcation and border
collision curves associated with the map and its first four iterates
in the $a$-$c$ plane for $b=0.9$, $n=5$ are depicted. It illustrates
the two possible types of fundamental bifurcation sequences and the
smooth change from the bifurcation sequence~(a) to the bifurcation
sequence~(b). The transition occurs when $(F^q)'(0^+)=1$ and the
saddle-node curve merges with the type~II border collision
curve. Furthermore it shows
\begin{itemize}
\item Every phase locked region is always bounded by a saddle-node curve on
  the right.
\item The border collision on the right is associated with a type~I
  border collision with an infinite derivative and hence cannot merge
  with the right saddle-node curve.
\item The border collision on the left is associated with a type~II
  border collision with a finite derivative. For the smaller values of
  $c$ (the derivative of $F_n'$ at $0^+$), this is also the left bound
  of the phase locked region.
\item For some phase locked regions, the left bound for the larger $c$
  values is formed by a saddle-node curve. This curve merges with the
  type~II border collision curve when $c$ decreases.
\item The function $F_n$ is convex for $c\in \left[0,\frac{bn}{2(n-1)}\right]
  = \left[0,0.5625\right]$ and only the bifurcation
  sequence~(a) occurs for those parameter values.
\end{itemize}

\subsection{Example: the Sinusoidal Threshold System (STS)}

Next we continue with our illustrative example of a threshold system, the STS
given by equations (\ref{eq:toymap1})-(\ref{eq:toymap3}).  As
discussed in section~\ref{sect:tangencies}, this system has an
associated circle map with a gap if {the flow is tangent to the upper
  threshold on contact. This will happen if} $h'(x) > \gamma$, i.e.
$ \alpha > \gamma$.  Thus increasing $\alpha$ smoothly changes the map
from continuous to a map with a gap. This enables us to see how the
familiar saddle-node bifurcation structure for smooth maps without
  gaps, as shown in Fig.~\ref{fig:tongues_saddlenodes}, transitions
to the bifurcation sequences: first (b) saddle-node bifurcation $\to$
border collision $\to$ border collision $\to$ saddle-node bifurcation,
and then (a) border collision $\to$ border collision $\to$ saddle-node
bifurcation.

Recall that the sufficient conditions for a saddle-node bifurcation
curve of a $(p,1)$ fixed point are $\beta = p\tilde\gamma$ or
$\beta = p\tilde\gamma -\frac\alpha\pi$. We will  see now why this
condition is only sufficient.
% These conditions are sufficient as it isn't verified that it occurs
% for a first intersection of the flow with the upper boundary.
%
An explicit expression for the infinite derivative type I border
collision of a $(p,1)$ fixed point can be derived {(recall that
  $\tilde \gamma = \frac{\gamma}{\gamma+1}$)},
$$
\alpha= \frac{\gamma^2 + 4\pi^2 \left (\tilde \gamma{p} - \beta\right )^2}{
  4 \pi \left( \tilde \gamma{p}- \beta \right) },
%$$
\quad \mbox{for} \quad
%$$
{\max\left(0\,,\,{p}\tilde \gamma - \frac{\gamma}{2 \pi}\right) \leq \beta < \tilde \gamma{p}}.
%\tilde \gamma - \frac{\alpha}{2 \pi} \leq \beta \leq \tilde \gamma.
$$
This curve is a monotonically increasing function of $\beta$ starting at
the minimum value
of $\beta=\tilde \gamma{p} - \frac{\gamma}{2\pi}$, $\alpha =\gamma$,
%\tilde \gamma - \frac{\alpha}{2\pi}, \alpha =\gamma,
where $t^n=t^{n+1}=1$ and asymptoting to the saddle-node line
$\beta=\tilde \gamma{p}$ where
$t^n=t^{n+1}=3/4$.

The finite derivative type II border collisions for $(p,1)$ fixed points satisfy
\begin{eqnarray*}
\cos 2 \pi x_b & = & \frac{\gamma}{\alpha}, \\
\sin 2 \pi x_c & = & \frac{2\pi}{\alpha} \left ( p \tilde \gamma -
                     \beta \right ) - 1, \\
%  \gd{\quad \sin \mbox{ instead of } \cos???}\\
\sin 2 \pi x_c -\sin 2 \pi x_b & = & \frac{2 \pi \gamma}{\alpha}
\left ( x_c - x_b \right ),
\end{eqnarray*}
where $x_c$ and $x_b$ are equivalent to the time points on the
periodic cycle indicated by $b$ and $c$ in Fig.~\ref{fig:gaps}. Note
that $x_b=x_c$ when $\alpha=\gamma$ and
$\beta=p \tilde \gamma - \gamma%\alpha
/2\pi$, at which point the type I and type II border collisions
coalesce. %Furthermore,
To find the intersection of the type~II border collision curve
  and the saddle-node curve,
we observe that $x_c=\frac14$ along the saddle-node curve
$\beta=p\tilde\gamma-\frac \alpha\pi$. Substituting this into the
equations gives the relation $1+\sqrt{1-q^2}=q(\frac\pi2 +
\arccos(q))$, where $q=\frac\gamma\alpha$. This has a unique solution
$q^*\in(0,1)$, $q^*\approx 0.725$. Hence when $\alpha = \frac{\gamma}{q^*}$ and
$\beta=p\tilde\gamma-\frac \gamma{q^*\pi}$, then the type II border
collision curve and the saddle-node curve collide and the border
collision curve terminates the saddle-node curve. This implies that at
$\alpha = \frac{\gamma}{q^*}$, the bifurcation sequence  sequence (b)
transitions to the sequence (a).

\begin{figure}

\centering{\includegraphics[scale=1.0]{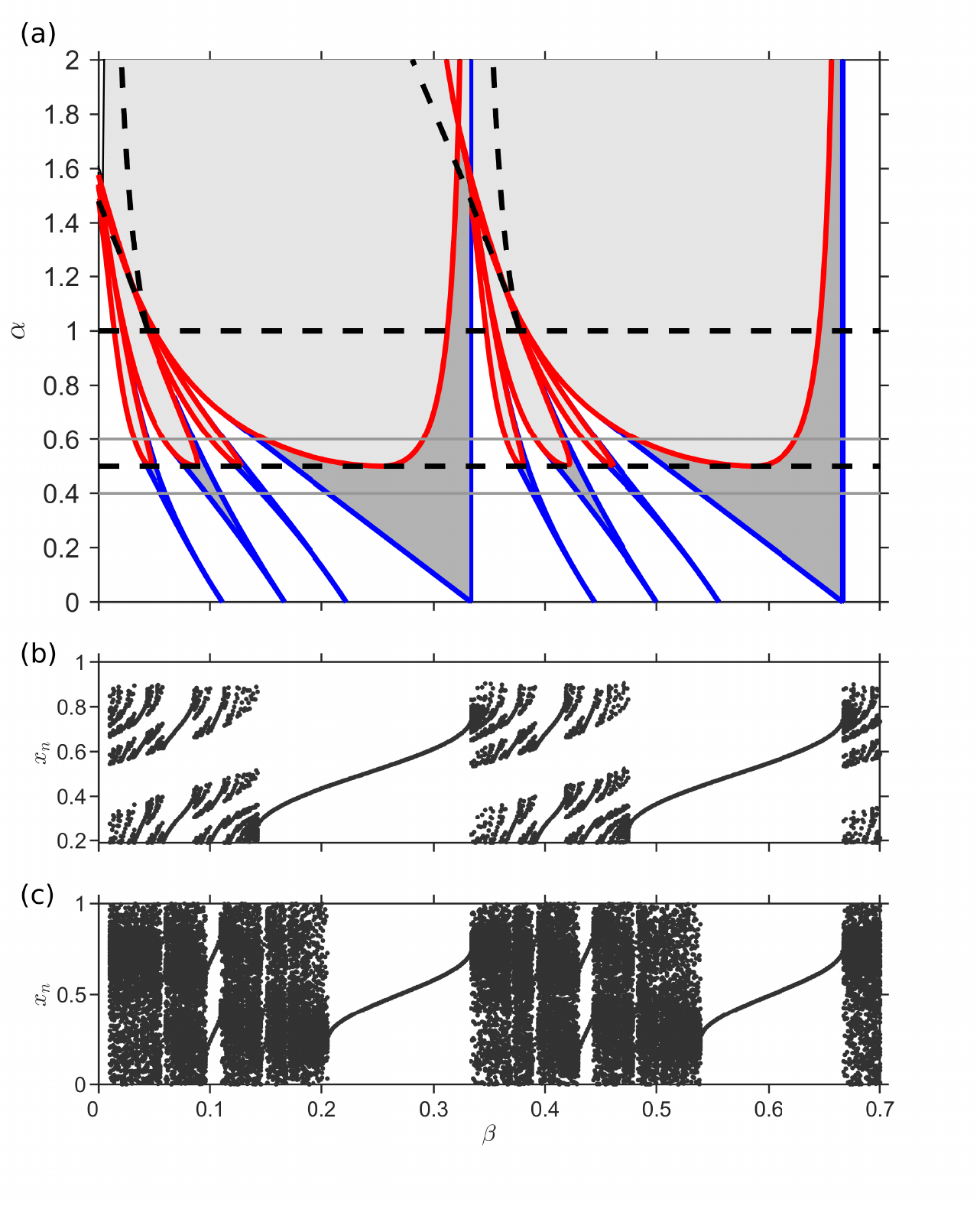}}

\caption{(a) Bifurcation set for $\gamma=0.5$ showing the relation
  between border collisions (red) and saddle-node bifurcations (blue).
  Border collisions to the left hand side of each minima are of type
  II and to the right hand side are of type I.  The dashed horizontal
  line at $\alpha=\gamma=0.5$ marks the transition from continuous to
  gap map.  The dashed horizontal line at $\alpha=1$ marks the
  transition from monotonicity to to nonmonotonicity. The dashed lines
  forming the `v' shape mark the transition from single to multiple
  gaps, see section~\ref{sect:nonmonotonic}.  For $\alpha<1$, the
  light/dark shaded regions correspond to regions of existence of
  one stable/a pair of fixed points. For $\alpha>1$ the map is nonmonotonic and the
  dynamics can be more complicated. In this region, period-doubling
  bifurcations also exist (not shown).
   (b) Bifurcation diagram
  showing stable solutions for $\gamma=0.5, \alpha=0.6$ (corresponding
  to the upper light grey line in (a)). The gaps in the map appear as
  bands of `forbidden' regions in the bifurcation diagram and result
  in the Cantor structure for quasi-periodic solutions. (c)
  Bifurcation diagram showing stable solutions for
  $\gamma=0.5, \alpha=0.4$ (corresponding to the lower light grey line
  in (a)).  The numerical bifurcation diagram has dark bands
  corresponding to the fact that there exist quasiperiodic solutions that densely
  fill the circle.}
\label{fig:bifdiagnogapandgap}
\end{figure}

Border collisions for general $(p,q)$ solutions can be found
numerically by solving the fixed point condition
(\ref{eq:STMpqsolution}) along with the requirement that the fixed
point occurs at the appropriate end point of the gap.
In Fig.~\ref{fig:bifdiagnogapandgap}(a) the bifurcation set for
$\gamma=0.5$ is shown for the first few iterates of the map.  Between
the lines $\alpha=\gamma=0.5$ and $\alpha=1$ the circle map is
monotonic with a gap. In this region, the border collisions form
u-shaped regions inside each tongue.  The left/right hand side of each
u-shaped region corresponds to the type II/type I border collision
bifurcations respectively. This illustrates how the sequences of
bifurcations seen in continuous circle maps transition to the sequences
seen in maps with gaps. The vertical derivative present in the
threshold maps implies that one side of the saddle-node tongues
is terminated by a border collision and the other side persists.

For $\alpha>1$, the map is non-monotonic. We consider the general
  transition from monotonic to non-monotonic maps in
  section~\ref{sect:nonmonotonic} and then continue with this example.

% We note that we highlighted a second difference between {smooth}
% monotonic maps without gaps and those with, namely that in the latter
% case quasiperiodic orbits no longer densely fill the circle but
% instead visit points in a Cantor set. This feature is evident in the
% numerically computed bifurcation diagrams shown in
% Fig.~\ref{fig:bifdiagnogapandgap}(b) and
% (c). {\textit{From~\cite{Granados_2017}, maps in $C^1$ or $C^0$
%     also can have a Cantor set as the omega limit set (Denjoy counter
%     example). Maybe reformulate the text above?}}

%%%%%%%%%%%%%%%%%%%%%%%%%%%%%%%%%%%%%%%%%%%%%%%%%%%%%%%%%%%%%%%%%%%%%%%%%%%%%%%%%%%%%%%%%
% Non-monotonic
%\input{Non-monotonic}
\section{Tangencies leading to non-monotonic maps}\label{sect:nonmonotonic}

We return to the general threshold maps. We have seen in
section~\ref{sect:tangencies} that tangencies of the up flow with the
upper threshold %(or the down flow with the lower threshold)
create discontinuities in the map.
% but the resulting circle maps remain monotonic.
In this section we will show that tangencies of the down flow with the
upper threshold lead to multiple pre-images (non-monotonicity) in the
down map $T_d$ (see Figure~\ref{fig:nonmonotonicity}).
This implies that
tangencies of the up flow with the lower threshold lead to
non-monotonicity in the up map $T_u$.  We will also discuss how and
when tangencies in the up or down flow imply non-monotonicity in the
full circle map.
% As illustrated
% in Fig.\ref{fig:nonmonotonicity}, if there are tangencies of the down
% flow with the upper threshold (or potentially the up map with lower
% threshold), then there exist some values for which $t^{n+1}$ has
% multiple pre-images, and the associated circle map is non-monotonic.
%
\begin{figure}
\centering{\includegraphics[scale=1.0]{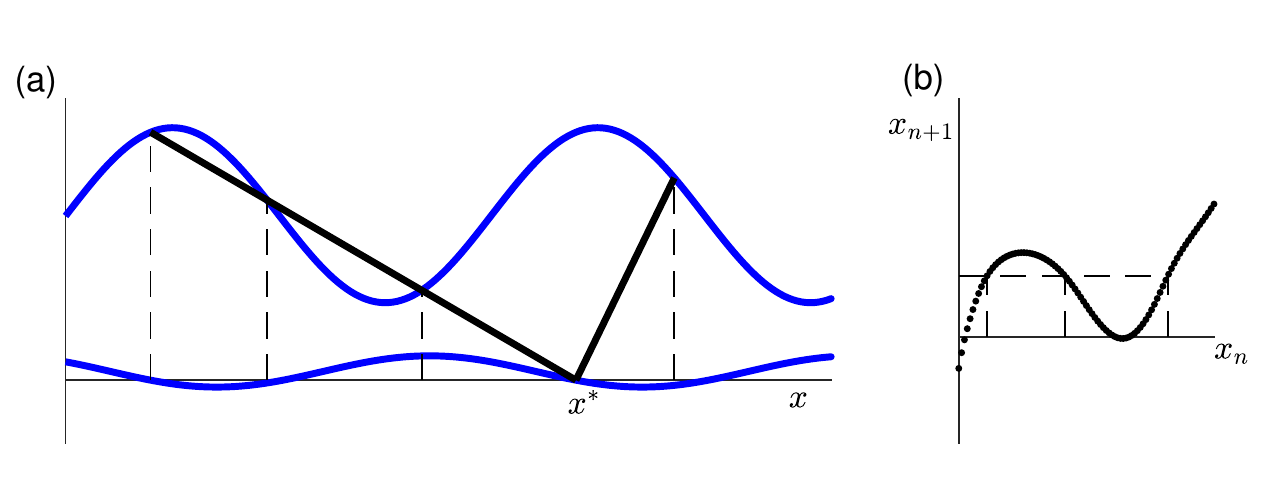}}
 \caption{(a) The point $x=0$ has multiple pre-images in the down map
   $T_d$, leading to non-monotonicity in the associated threshold
   system circle map, see (b).}
 \label{fig:nonmonotonicity}
 \end{figure}
Finally we will discuss how simultaneous tangencies in the up and
down maps lead to the combination of gaps and nonmonotonicity.  This
can lead to the presence of multiple gaps (see
Fig.~\ref{fig:multiplegaps}) and give rise to a codimension~2
bifurcation which organises the local bifurcation
structure. %We discuss the mechanism by which this happens and
We illustrate the mechanism and the consequent bifurcations with the
STS {example}.
% In this section we first show that generically the non-monotonicity results
% in maps that are locally quadratic, which naturally leads to the
% possibilities of period-doubling bifurcations and chaos.
% However, there are further consequences as the combination

\begin{figure}
\centering{\includegraphics[scale=1.0]{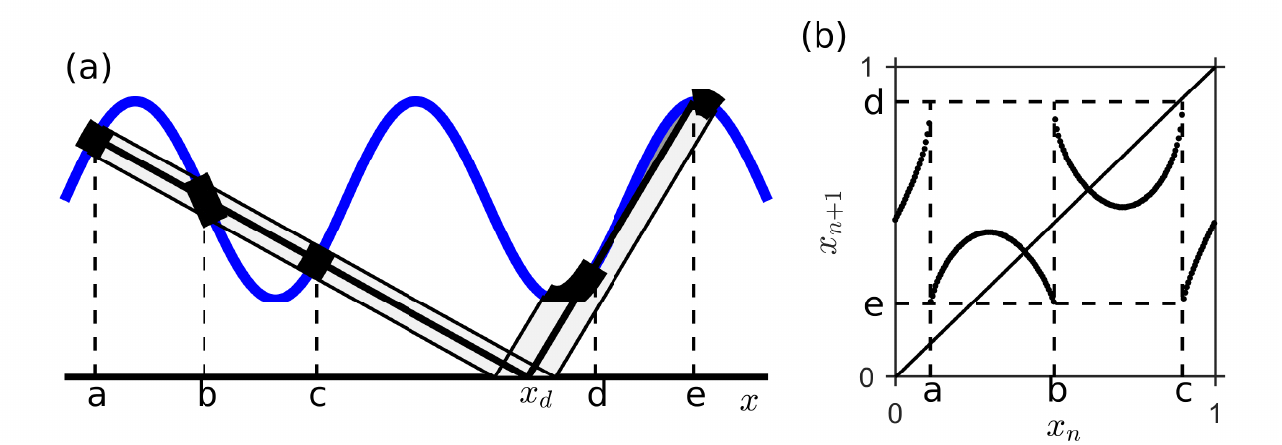}}
\caption{A tangency point between the down flow and the upper
  threshold leads to multiple pre-images, as shown in (a) for the STS
  ($\alpha=4, \beta=0.5, \gamma=3$). If there is also a tangency
  between the up flow and the upper threshold, then the corresponding
  threshold map has multiple gaps, where each gap has the same size
  with an infinite derivative on one side (at $x_{n+1}=d$) and a finite derivative on
  the other (at $x_{n+1} = e$).}
\label{fig:multiplegaps}
\end{figure}

\subsection{Existence of nonmonotonicity}

As in section~\ref{sect:tangencies}, we consider a family of
parameterised threshold maps with $\mathbb{P}$ the parameter
space. The threshold maps are the composition of two maps: the down
map $T_d:\mathbb R\times \mathbb P\to \mathbb R$ from the upper
boundary to the lower boundary, and the up map
$T_u:\mathbb R \times \mathbb P\to \mathbb R$ from the lower boundary
to the upper boundary.
A map is monotonic if every point in its range has exactly one
pre-image. By using the backward flow, we can find the pre-images of
the down map. Let's   consider the function
\begin{equation}
  \label{eq:tW}
\tW(\tau,x,\mu) = \psi_{-\tau}(g(x,\mu),\mu)-h(x-\tau),
\end{equation}
i.e., $\tW$ is very similar to~$W$ as defined in~\ref{eq:basic0},
but uses the backward down flow starting at the lower threshold $g(x,\mu)$.
If $(\tau^*,x^*,\mu^*)$ satisfies $\tW(\tau^*,x^*,\mu^*)=0$, then
$x^*-\tau^*$ is a pre-image of $x^*$ for the down map~$T_d$.  Using the
convention for derivatives from section~\ref{sect:tangencies}, if also
$\tW_1(\tau^*,x^*,\mu^*)=0$, then there is a tangency between the down
flow~$\psi_\tau(h(x,\mu),\mu)$ and the upper threshold $h(x,\mu)$ at
$x=x^*-\tau^*$, $\mu=\mu^*$ and $\tau=0$. Due to the similarity
between $\tW$ and $W$, the results of section~\ref{sect:tangencies}
give the local behaviour near a pre-image. Let $(\tau^*,x^*,\mu^*)$
satisfy $\tW(\tau^*,x^*,\mu^*)=0$, hence $x^*-\tau^*$ is a pre-image
for $x^*$ under $T_d$. Assume the non-degeneracy conditions
$\tW_2(\tau^*,x^*,\mu^*)\neq 0$ and $\tW_3(\tau^*,x^*,\mu^*)\neq 0$,
then we have the following results.
\begin{itemize}
\item If $\tW_1(\tau^*,x^*,\mu^*)\neq 0$, then for $(x,\mu)$ nearby
  $(x^*,\mu^*)$ there is a locally unique pre-image.
\item If $\tW_1(\tau^*,x^*,\mu^*)=0$ and
  $\tW_{11}(\tau^*,x^*,\mu^*)\neq 0$, then there is a fold along a
  curve in the $(x,\mu )$-plane given by $\tW_2x+\tW_3\mu =0$ in
  lowest order. The fold has again two interpretations: it represents
  a unique pre-image along the fold line.
  %$x-\tau(\mu) = x-\tau^* + O(\mu)$.
  And in the threshold system it
  gives also a persisting simple tangency between the upwards
  trajectory and the lower threshold at this point. When
  $\tW_{11}(\tW_2x+\tW_3\mu)< 0$ then there are two pre-images, one
  less and one greater than $x^*-\tau^*$. Both pre-images are relevant
  for the map $T_d$, which has a turning point at the unique
  pre-images on the fold line, i.e, at the tangency points,
  see Fig.\ref{fig:nonmonotonicity}.
\item If $\tW_1(\tau^*,x^*,\mu^*)=0=\tW_{11}(\tau^*,x^*,\mu^*)$ and
  $\tW_{111}(\tau^*,x^*,\mu^*)\neq 0$, then again generically we have
  cusp unfolding and locally there is a change in monotonicity with
  two turning points emerging in the map~$T_d$.
\end{itemize}
As $T_d$ is periodic, this local analysis shows that globally the map
$T_d$ always has an even number of tangency points and an odd number
of pre-images (counting multiplicity at the degenerate points).

% \begin{lemma}
%   \acs{\textit{Re-word?}}  Suppose there exists a parameter
%   $\tilde \mu$ such that on one side of $\tilde \mu$, $\mu>\tilde \mu$
%   say, there exists a contact point between the down flow and the
%   upper threshold at a tangency, and on the other side of $\tilde \mu$
%   there are no tangencies {at the contact points} of the down flow
%   and the upper threshold.

% Then for $\mu<\tilde \mu$ the map is locally invertible,
% and for $\mu>\tilde \mu$ the map has two locally quadratic
% regions, one of which is a local minimum and the other is a
% local maximum.

% There is an equivalent result for gaps arising from
% tangencies of the up flow with the lower threshold.
% \end{lemma}

% This can be proved using essentially the same ideas as for lemma
% \ref{lemma:gaps}. {Note that now both tangencies with the upper
%   threshold are relevant. Indeed they give the two extrema of the map.}

%\subsection{Consequences of non-monotonicity for the dynamics}
\subsection{Simultaneous tangencies}

We have now shown for the up and down maps that tangencies between the
flow and the nearby threshold correspond to non-monotonicity and
tangencies between the flow and the opposite threshold correspond to
discontinuities in those maps (section~\ref{sect:tangencies}). The
threshold map is a composition of the up and down maps, hence these
tangencies will influence the monotonicity and continuity of the
threshold map. The derivative of the threshold map is
\[
  (T_u\circ T_d)'(x) = T_u'(T_d(x))\, T'_d(x),
\]
thus a tangency between the down flow and the upper threshold leads to
a turning point in the threshold map. And a tangency between the
up flow and the lower threshold leads to a turning point in the
threshold map if this tangency occurs at a point that is in the range
of the up map.

If both the up and down map have tangencies with the upper threshold,
then the result will be a non-monotonic, discontinuous map. The
presence of both gaps and non-monotonicity can also lead to further
structural changes in the map: a transition from a single gap to three
(or more) gaps.  The transition to multiple gaps occurs as follows
(see also Figure~\ref{fig:multiplegaps}). Suppose that the up flow is
tangent to the upper threshold at a first intersection point at $x=d$.
This leads to a gap in the up map~$T_u$, say at $x=x_d$, i.e,
$T_d(x^-_d)=d$ and $T_d(x^+_d)=e$ for some $e>d$.  When the down map
is monotonic (i.e. has no tangencies between the down flow and the
upper threshold), the point $x=x_d$ will have a single pre-image in
the down map and there will be a single gap in the threshold map.
When the down map is non-monotonic, the point $x=x_d$ can have three
(or more) pre-images, as illustrated in Fig.~\ref{fig:multiplegaps}.

In the corresponding circle map, there will be a gap associated with
each of these pre-images. In each case the gap arises from the same
tangency, thus the size and the qualitative nature of the gap will be
preserved.  The unbounded derivative can occur either to the left or
the right of the gap depending on the slope of the upper threshold at
the pre-image.

The transition point between one and three gaps occurs
when the down map has a tangency with the upper
threshold and this tangency point is mapped into $x=x_d$.
%then there
%is a transition between one and three pre-images in the down map and
%in the threshold map this gives a transition between one and three
%gaps.
In the notation of Figure~\ref{fig:multiplegaps}, at this
special point, we have $b=c$ and the threshold map maps $b$ into
$d$. This is an isolated point in the map as all points nearby $b$ get
mapped nearby~$e$. The finite derivate at $e$ vanishes and there is no
derivative at $d$ as it is an isolated point. How the creation of
multiple gaps plays out in the STS circle map is illustrated in Fig.~\ref{fig:wedge}.
The two points at which tangencies with the upper threshold map
into $x=x_d$ correspond to a local maximum and a local minimum
of the circle map. These tangencies are mapped into $x=x_d$ when
the local maximum coincides with the
infinite derivative (see Fig.~\ref{fig:wedge}(a))
or the local minimum coincides with the finite derivative
(see Fig.~\ref{fig:wedge}(b)). The isolated point in both cases is
marked in orange.

Tangencies in both the up and down flow with the lower threshold do
not lead to multiple gaps. This apparent asymmetry is a consequence of
the fact that we have considered the map from the upper threshold to
the upper threshold ($T_u \circ T_d$). Generically, the tangency of
the down map with the lower threshold creates a gap, this gap persists
under the action of the up map, and hence creates a gap in the
threshold map. Nonmonotonicity of the up map (tangencies of the up
flow with the lower threshold) do not affect the pre-images of the
gap. In this case, the only mechanism in which multiple gaps can be
created is by non-monotonicity of the down map, i.e, the down flow
being tangent simultaneously to the upper and lower threshold. Thus
the multiple gaps in the threshold map reflect the multiple tangencies
in the down map.  In section~\ref{sect:cod2} we discuss further the
influence of the order of composition on the apparent structure of the
map.

The consequences of non-monotonicity for bifurcations in the standard
circle map are discussed in~\cite{Mackay_1986}. All the typical
features of period-doubling bifurcations and the associated transition
to chaos can be seen in threshold systems {too}.

\subsection{Example: The sinusoidal threshold system (STS)}

Consider the STS given by equations
  (\ref{eq:toymap1})-(\ref{eq:toymap3}).  As the lower threshold is
  flat, there is no tangency between the up flow and the lower
  threshold, hence the up map~$T_d$ %
% t_d^n\mapsto t^{n+1}$
is monotonic and the down map~$T_d$ is continuous.
Thus the STS map is monotonic if and only if the down map~$T_d$ %
% t^n\mapsto t_d^n$
is monotonic, i.e., if and only if $h'(x) \leq 1$ for all~$x$, which
is equivalent to $0<\alpha \leq 1$.  Transitions from one to three
gaps occur when both the up and down flow have tangencies with the
upper threshold and the tangency point of the down flow with the upper
threshold is mapped into the $T_u$ pre-image of the tangency point of
the up flow with the upper threshold.
% when the end points of the gap occur at a local minimum or maximum of
% the map. In this example, gaps in the map are formed by tangencies of
% the up flow with the upper threshold and extrema of the map are occur
% at tangencies of the down flow with the upper threshold. Thus a
% transition from one to three gaps occurs when both the up and down
% flow have tangencies with the upper threshold.

These transitions can be computed numerically and are shown for the
particular case $\gamma=0.5$ {by the v-shaped curves formed by the
  dashed lines in Fig.~\ref{fig:bifdiagnogapandgap}(a) and}
Fig.~\ref{fig:wedge}(c).
Even though there are three gaps, no new
border collision curves will be formed. This follows from the
observation that the fixed points and bifurcation curves of the maps
$T_u\circ T_d$ (mapping upper threshold into upper threshold) and
$T_d\circ T_u$ (mapping lower threshold into lower threshold) are
equivalent. The latter map is the composition of a non-monotonic,
continuous map acting on a monotonic, discontinuous map. So the gap
has a unique pre-image, implying that there are at most two border
collision curves, see also the section below.

\begin{figure}
\begin{center}
\centering{\includegraphics[scale=1.0]{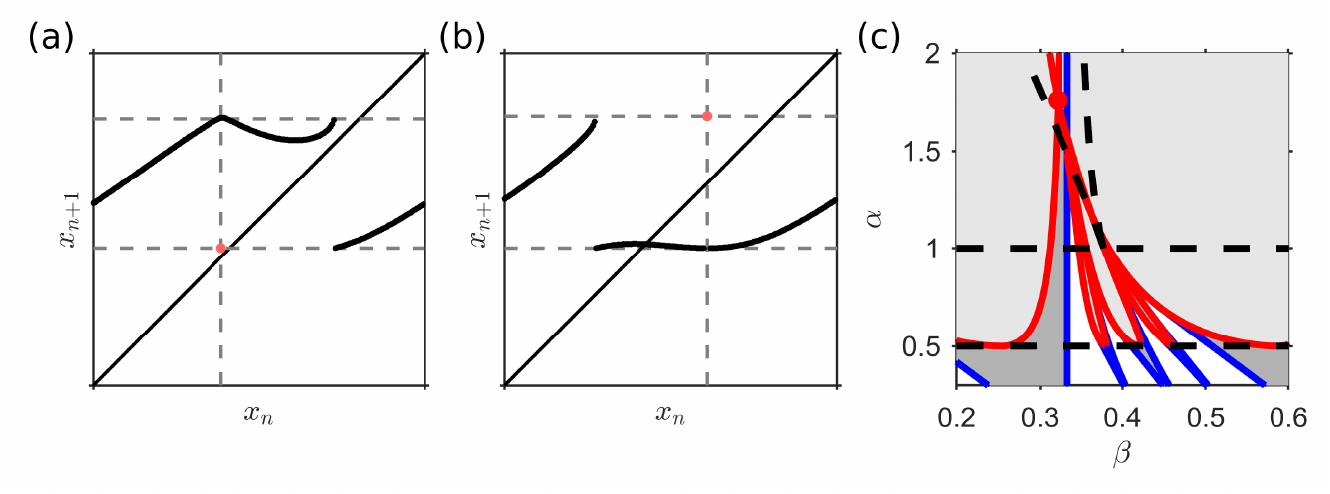}}
\end{center}
\caption{
(a) Map on the left-hand edge of the v-shaped wedge
for $\alpha=1.3$, $\beta=0.3508$, $\gamma=0.5$. This left-hand edge corresponds
to the point when the local maximum coincides with the side of the gap
with infinite derivative. The orange dot denotes the isolated point in
the map.
(b) Map on the right-hand edge of the v-shaped wedge
for $\alpha=1.3$,  $\beta=0.3653$, $\gamma=0.5$.
(c) Bifurcation set for the
STS for $\gamma=0.5$ showing a blow-up of the v-shaped region.
}
\label{fig:wedge}
\end{figure}

Fig.~\ref{fig:wedge}(c) shows that the type I border collisions for the
$(1,1)$-tongue and the type~II border collisions for the
$(2,1)$-tongue cross in the three-gap region (inside the v-shaped
region).
The map at this point is shown in Fig.~\ref{fig:cod2}(c). The two
border collision points are denoted by $c_0$ (type II border
collision) and $c_1$ (type I border collision), i.e., the tangency
between the up flow and the upper threshold is at $c_1$.  This implies
that there is some $x_d$ such that $T_u(x^-_d)=c_1$, $T_u(x^+_d)=c_0$,
and $T_d(c_0)=x_d=T_d(c_1)$.
Many of the border collisions from the intermediate tongues appear to
converge on this crossing point, suggesting that it forms a
codimension two point that organises the local bifurcation
structure. In the next section we will show that this is indeed the
case.% by using~\cite[\S 7.1.1]{Granados_2017}.

\begin{figure}
\centering{\includegraphics[scale=0.9]{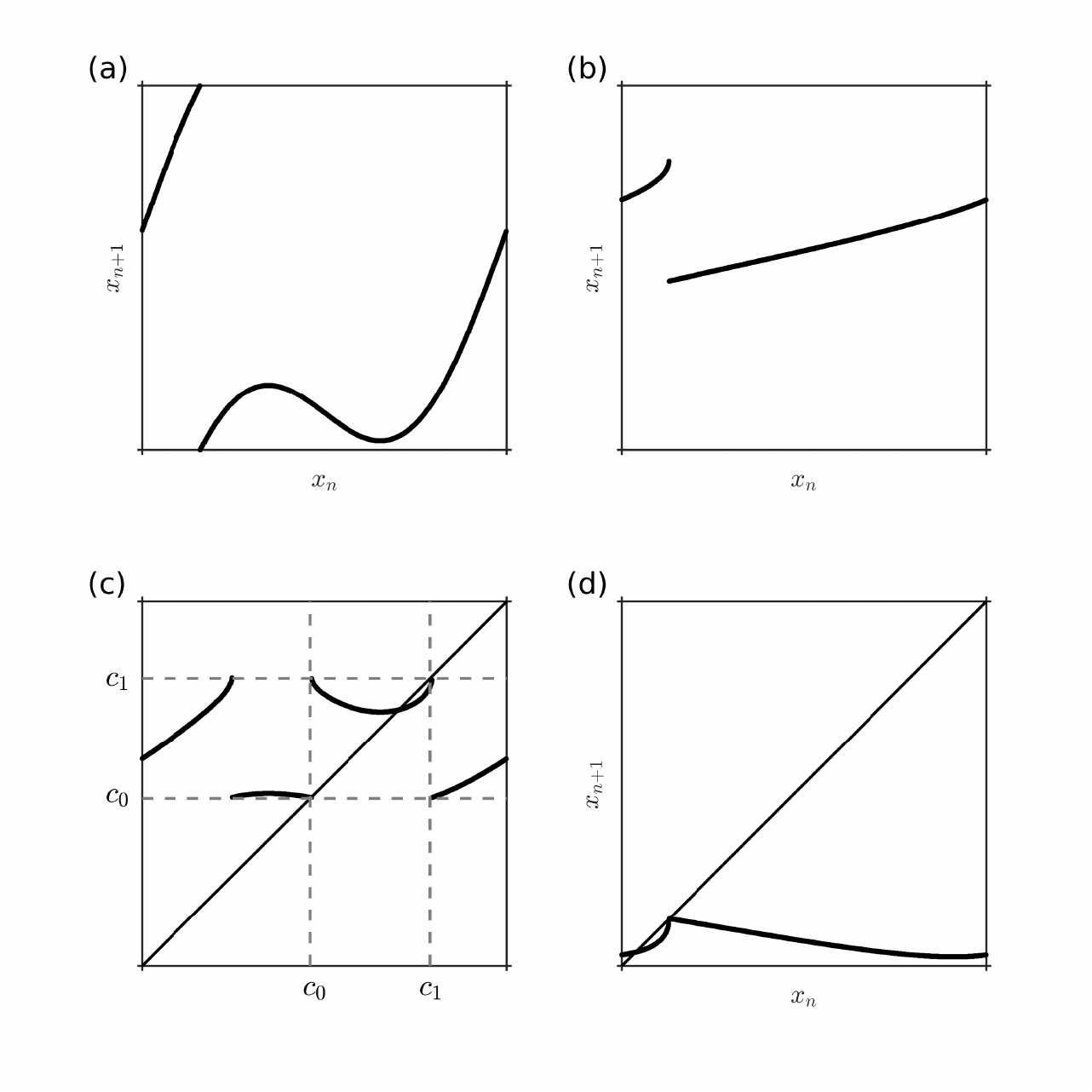}}
\caption{Maps at the intersection point of the border collisions, the point
which is marked by a red dot in Fig.~\ref{fig:wedge}(c).
(a) The down map $T_d$ which is non-monotonic because of a tangency
of the down flow with the upper threshold. (b) The up map
$T_u$ which contains a gap as a consequence of a tangency
of the up flow with the upper threshold. (c)
 $T_d \circ T_u$ (d) $T_u \circ T_d$.
}
\label{fig:cod2}
\end{figure}

Though we will not go into the details here, we note that explicit
expressions can be derived for the first period-doubling bifurcation
for $(p,1)$ fixed points, giving rise to fixed points $(2p,2)$.  The
case $\gamma=3$ is particularly interesting since at this value, the
period-doubling and type~I border collision curves coincide.

% The generic structure and consequent dynamics of such a codimension two point are
% analysed in further detail in the next section.

%\item Note interesting codimension two point. What is happening to the saddle-node bifurcation?
%\item Think about transition to more gaps. Can have a sequence of transitions 1 to 3 to 5 etc.
%Any more cod 2 points (think not)?
%\end{itemize}

%%%%%%%%%%%%%%%%%%%%%%%%%%%%%%%%%%%%%%%%%%%%%%%%%%%%%%%%%%%%%%%%%%%%%%%%%%%%%%%%%%%%%%%%%
%Codimension two bifurcations
%\input{Cod_two}

\subsection{A codimension two bifurcation}\label{sect:cod2}

In this section we conjecture %show
that if the up flow has a generic
tangency with the upper threshold and the down map maps both end points
of this gap into the pre-image of the up map associated with
the tangency, then the corresponding codimension~2 border collision
bifurcation point is an organising centre for the local
bifurcations. Introducing the parameter vector $\bmu\in\mathbb{R}^2$,
such codimension~2 point is characterised by the following
properties. There exist $c_0$, $c_1$, $c_0\neq c_1$ such that
% \begin{align*}
% T_u(c_1;\bzero)&=T_u(c_0;\bzero)=:x_0; \quad
% T'_u(c_1;\bzero)\,T'_u(c_0;\bzero)<0;\\ %\quad
%    T_d(x_0^-;\bzero)&=c_0; \,\,T_d(x_0^-;\bzero)=c_1.
% \end{align*}
% or (swapping down and up)
 \begin{eqnarray*}
    %     \exists_{c_0,c_1, c_0\neq c_1}[\,
%\begin{align*}
  T_u(x_0^-;\bzero)&=c_0; \,\,T_u(x_0^+;\bzero)=c_1;\\
  T_d(c_1;\bzero)&=T_d(c_0;\bzero)=:x_0; \quad
       T'_d(c_1;\bzero)\,T'_d(c_0;\bzero)<0. %\quad
%  \end{align*}
 \end{eqnarray*}
% Focusing on the first case, i
  This implies that the threshold map~$T_u\circ T_d$ has multiple gaps
  and there are two simultaneously two border collisions: one type I
  border collision at~$c_0$ and one type II border collision
  at~$c_1$. To analyse this border collision, it is convenient to consider
  the map from the lower threshold to the lower threshold, i.e.,
  $T_d\circ T_u$. This map has the same fixed points and bifurcations
  as the threshold map~$T_u\circ T_d$. Under the conditions above, the
  border collisions in the map $T_d\circ T_u$ collide when
  $\bmu=\bzero$. The map is continuous at $x_0$ and the derivatives on
  each side have opposite signs (with one of them having a square root
  singularity).

  In section~\ref{sect:tangencies} it is shown that the gap in $T_u$
  persists for $\bmu$ small and that generically the derivative of
  $T_u$ nearby the gap has the same sign at both sides of the gap
  (with one of them having a square root singularity).  For
  $\bmu\neq\bzero$, the discontinuity in $T_u$ leads to a
  discontinuity in the threshold map $T_d\circ T_u$ with the
  derivatives at each side of the gap still have opposite signs. Hence
  the parameter plane nearby $\bmu=\bzero$ can be divided into four
  regions which are such that the threshold map has two solutions in
  one region, one solution in two regions, and no solutions in one
  region.  In~\cite[\S 7.1.1]{Granados_2017}, it is shown that if the
  local derivatives are less than~1 (i.e., the map is contracting)
  such maps are organising centres for the local bifurcations.  It is
  also noted that this bifurcation point is equivalent to the
  \emph{gluing bifurcation} in~\cite{Gambaudo_1984,Gambaudo_1988} and
  \emph{big bang bifurcation} in~\cite{Avrutin_2006}. Although we do
  not satisfy the condition that the derivatives are less than~1
  (there is a square root singularity at one of the end points), we
  still see the organising centre in the bifurcation diagram. Avrutin
  et al \cite{Avrutin_2010} have studied maps on the real line with
  similar singularities in the derivative, although our local
    behaviour does not seem to be one of the cases they study in
    detail. Our behaviour looks more like a 1D Nordmark map at the
    grazing point.

\section{Other mechanisms: Cherry flows}\label{sect:Cherryflow}
Circle maps arise naturally in other contexts. If two oscillators interact then a
lowest order model might relate the evolution of the phase of each oscillator. 
In this case the natural phase space is the torus (one angle for each phase)
leading to a differential equation on the torus. Examples include neuronal
models such as the Kuramoto equations \cite{Kopell_2002} and  models of
breathing patterns \cite{Baesens_2013b}.

Suppose that a flow on the torus has a global cross-section transverse to the flow.
Then the return map on the global section is a circle map as
discussed in previous sections. There are two natural classes \cite{Baesens_2013a}.
In a  Poincar\'e flow this map is continuous and monotonic, so the
classic results about the existence of rotation numbers and the dichotomy of
dynamics depending on whether the rotation number is rational or irrational hold.
On the other hand, a Cherry flow \cite{Baesens_2013a,Palis_1982,Palmisano_2015} has at least
one unstable stationary point and one saddle, which can create a return map
which is monotonic and with a discontinuity. Such maps
have a well-defined rotation number and for continuous perturbations of the
defining vector field this rotation number varies continuously (see section~\ref{section:intro} and \cite{Rhodes_1986,Rhodes_1991}. In particular 
if the family of maps has parameter values which have rotation numbers that
are different, then there are parameters with irrational rotation numbers.
Since the image of the cross-section is not surjective because of the
discontinuity this is a natural way to construct a Denjoy counterexample
(a map with an irrational rotation number but no dense orbits).

Although Cherry flows are classic examples from geometric dynamics \cite{Palis_1982}, the
transition from a Poincar\'e flow to a Cherry flow has not been discussed
in the literature. In this section we give a brief account of the
scalings predicted by a theoretical model of this transition and
describe a piecewise smooth example. We will show that
\begin{itemize}
\item the size of the gap is finite at the transition point, so the
transition is discontinuous; and
\item the slope of the map tends to infinity at both boundaries of the jump.
\end{itemize}

A pair of stationary points can be created in a saddle-node bifurcation.
Suppose that $\tilde\mu$ is a real parameter and that if $\tilde\mu >0$
there are no stationary points of the flow, whilst if $\tilde\mu<0$
there is an unstable stationary point and a saddle. Then there are local
coordinates $(\xi ,\eta )$ such that in a neighbourhood
of $(\xi, \eta ,\tilde\mu )=(0,0,0)$ the leading order terms of the
differential equation are
\begin{equation}\label{eq:locsn}\begin{array}{rl}
\dot\xi & = \mu +\xi^2\\
\dot\eta & = \lambda \eta\end{array}
\end{equation}
with $\lambda>0$ and $\mu$ is a rescaled version of the original
parameter $\tilde\mu$. We would like to derive a leading order return
map through a neighbourhood of the origin from $\xi<0$ to $\xi >0$
which can then be composed with the standard return maps for Poincar\'e
type flows away from this singularity to obtain a theoretical model
of the global return map for the transition from a Poincar\'e flow to a Cherry flow.

In the classic form (\ref{eq:locsn}), the unstable manifolds of the
stationary points are vertical and so it is not possible to define
a return map from positive to negative $\xi$. This suggests that
for this problem we should add a further change of coordinates
\begin{equation}\label{eq:coords}
x=\xi+a\eta^2 , \quad y=\eta, \quad (a>0)
\end{equation}
so that in these new coordinates the unstable manifold of the saddle-node
stationary point at $\mu=0$ is $x=ay^2$, $\xi =0$ in (\ref{eq:coords}),
making a return map from negative $x$ to positive $x$ possible even
if $\mu <0$. Another way of seeing this is to claim that generically
the unstable manifolds will be quadratic at the saddle-node bifurcation,
and this modification of coordinates has the effect of making the
unstable manifolds quadratic without complicating the underlying dynamics.

In the new coordinates (\ref{eq:locsn}) becomes
\begin{equation}\label{eq:locsna}\begin{array}{rl}
\dot x & = \mu +x^2+2a\lambda y^2-2axy^2+a^2y^4\\
\dot y & = \lambda y.\end{array}
\end{equation}
As with (\ref{eq:locsn}), (\ref{eq:locsna}) is the leading order
approximation of the vector field in a neighbourhood ${\mathcal N}$
of the origin in phase space and parameter space,
i.e. $|x|^2+|y|^2+|\mu |^2<\epsilon^2$ for some small $\epsilon >0$.

Fix $\epsilon >0$ and ${\mathcal N}$ as above and let $k$ be a
constant, $0<k<1$ to be determined. Our goal is to derive a
return map of (\ref{eq:locsna}) from $x=-k\epsilon$ to $x=k\epsilon$ in ${\mathcal N}$.

If the initial condition is $(-k\epsilon, y_0)$ in ${\mathcal N}$
then this corresponds to $(\xi_0 ,\eta_0 )= (-k\epsilon -ay_0^2, y_0)$
and since $a>0$, $\xi_0<0$.

Suppose that $\mu >0$, so we can write
\begin{equation}\label{eq:posmu}
\mu=\sigma^2, \quad \sigma >0 .
\end{equation}
 Solutions to (\ref{eq:locsn}) are
\begin{equation}\label{eq:sl}
\xi =\sigma\tan (\sigma t+C), \quad \eta =\eta_0\exp(\lambda t),
\end{equation}
or
\begin{equation}\label{eq:sol}
x =\sigma\tan (\sigma t+C)+ay_0^2\exp(2\lambda t), \quad  y=y_0\exp(\lambda t).
\end{equation}

The initial condition $(-k\epsilon, y_0)$ implies that
\[
\sigma\tan C+ay_0^2=-k\epsilon
\]
and so as $\sigma \to 0$
\begin{equation}\label{eq:C}
C=-\frac{\pi}{2}+\frac{\sigma}{k\epsilon +ay_0^2} +O(\sigma^{2}).
\end{equation}
Provided it stays in ${\mathcal N}$ this solution intersects $x=k\epsilon$
after time $T$ given by
 \begin{equation}\label{eq:bal}
\sigma\tan (\sigma T+C)+ay_0^2\exp(2\lambda T)=k\epsilon .
\end{equation}
If $|y_0|$ is sufficiently small so that the first term on the left
hand side  of (\ref{eq:bal}) dominates the second term this leaves
\[
\sigma\tan (\sigma T+C)\approx k\epsilon ,
\]
so
\begin{equation}\label{eq:T}
T\approx \frac{\pi}{\sigma} +O(1),
\end{equation}
and $T\to \infty$ as $\sigma \downarrow 0$.

This approximation holds in an exponentially small region of parameter space with
\[
|y_0|\ll \epsilon\exp (-\lambda T).
\]
However, in this very small neighbourhood of $y_0=0$ the return map
through the region where the saddle-node bifurcation is about to take place is approximately
\[
y\to \exp(\lambda T)y\approx (\textrm{e}^{\lambda\pi})^{\frac{1}{\sigma}}y .
\]
In other words there is a small neighbourhood on which the
slope $s$ of the map is very steep and for $\mu\to 0^+$ tends to infinity with
\begin{equation}
\label{eq:slopscal}
\log s \approx \frac{\lambda\pi}{\sigma}=\frac{\lambda\pi}{\sqrt{\mu}} .
\end{equation}
Note that the constant $\lambda\pi$ is determined by the normal form (\ref{eq:locsn}), and in general $\log s\approx \kappa /\sqrt{\mu}$ for some constant $\kappa$.

Now suppose that $\mu<0$, so
\begin{equation}\label{eq:negmu}
\mu =-\sigma^2, \quad \sigma \ge 0.
\end{equation}
The derivation of the return map is more standard in this case. By construction
there are stationary points at $(\pm \sigma , 0)$ and $(-\sigma ,0)$ is a
saddle. The unstable manifold of the saddle in $(x,y)$ coordinates is the curve
\[
x=-\sigma +ay^2
\]
and so this intersects
$x=k\epsilon$ at $y=\pm\frac{1}{a}\sqrt{k\epsilon+\sigma }$.
Most importantly, this is non-zero for all $\sigma \ge 0$. In other
words the map develops a non-zero discontinuity at the bifurcation point $\mu=0$.
Moreover, standard analysis (e.g.\ \cite{Glendinning_1994}) close to the saddle shows that the slope of the
return map at the discontinuity tends to infinity as the leading order
non-constant term of the return map is
\[
C_1|y|^\alpha, \quad \alpha =2\sqrt{|\mu |}/\lambda <1.
\]

To summarise the theoretical predictions we have
\begin{itemize}
\item if $\mu >0$ then as $\mu \downarrow 0$ the global return map
develops an exponentially small region on which the slope $s$ of
the return map grows large and scales with $\log s\approx \frac{\kappa}{\sqrt{\mu}}$ for some constant $\kappa$;
\item if $\mu\le 0$ then there is a finite discontinuity and if $\mu <0$
then the slope of the map tends to infinity at the discontinuity.
\end{itemize}

The system (\ref{eq:locsna}) can be embedded in a global flow to
create a piecewise smooth example of the transition from a Poincar\'e
flow to a Cherry flow that can be analyzed numerically. Note that there
is no reason why a $C^\infty$ interpolating function could not be used to
smooth out the discontinuities in the defining flow, but this would not
add significantly to the discussion here.

We will define three vector fields and then show that they can be used in
different regions of the phase space ${\mathbb T}^2=[0,1]^2$ to define
a continuous flow on the torus with the desired properties. In
\[ A =\{(x,y)~|~\textstyle{\frac{3}{8}}<x<\textstyle{\frac{5}{8}},~\textstyle{\frac{3}{8}}<y<\textstyle{\frac{5}{8}}\}
\]
we use the saddle-node bifurcation (\ref{eq:locsna}) transformed to the
centre of the square:
\begin{equation}\label{Um}
U_\mu (x,y)=\left(\begin{array}{c}\mu +(x-\textstyle{\frac{1}{2}})^2+2a(\lambda -(x-\textstyle{\frac{1}{2}}))(y-\textstyle{\frac{1}{2}})^2+a^2(y-\textstyle{\frac{1}{2}})^4\\
 \lambda (y-\textstyle{\frac{1}{2}}).\end{array}\right)\end{equation}
Below the square $A$, in
\[ B =\{(x,y)~|~\textstyle{\frac{3}{8}}<x<\textstyle{\frac{5}{8}},~0<y<\textstyle{\frac{3}{8}}\}
\]
define
\begin{equation}\label{V}
V(x,y)=\left(\begin{array}{c}1 \\
b-y \end{array}\right), \quad 0<b<\textstyle{\frac{3}{8}}.\end{equation}
Finally, in the remainder of the torus
\[
C=[0,1]^2\backslash (A\cup B)
\]
define
\begin{equation}\label{W}
W(x,y)=\left(\begin{array}{c}1 \\
c\end{array}\right), \quad c>0.\end{equation}

The constants in the equations now need to be restricted so that trajectories 
cross the boundaries between the regions in the same direction, which
implies that solutions can be continuously extended across these
boundaries (technically this implies that the system has unique Carath\'eodory 
solutions for all initial conditions \cite{Filippov}).

Consider first the boundary between regions $B$ and $C$. This is two
vertical line segments and one horizontal line segment. On the vertical
lines $\dot x =1$ in both (\ref{V}) and (\ref{W}), so both flows are
transverse to these surfaces in the same direction. On the horizontal line $y=0$, so
$\dot y = c$ from below using (\ref{W}) and $\dot y=b$ from above using (\ref{V}),
so again the flow is transverse to the boundary and in the same direction as $b,c >0$.

There is only one boundary between regions $A$ and $B$: the line
segment with $y=\frac{3}{8}$ and $\textstyle{\frac{3}{8}}<x<\textstyle{\frac{5}{8}}$.
If $y=\frac{3}{8}$ then the flow in $A$ has $\dot y=-\frac{1}{8}\lambda<0$
whilst the flow in $B$ has $\dot y=b-\frac{3}{8}$, so if
\begin{equation}\label{eq:bcond}
0<b<\textstyle{\frac{3}{8}}
\end{equation}
then $\dot y<0$ from below, and hence once again the flow is transverse to the 
boundary line segment and crosses it in the same direction from each side.

The horizontal boundary between $A$ and $C$ has $\dot y >0$ on both sides,
but the vertical boundaries require further constraints. On these
boundaries $|y-\textstyle{\frac{1}{2}}|=\frac{1}{8}$ and so
if  $z=x-\frac{1}{2}$ then from (\ref{Um}) on these boundaries approached from $A$
\[
\dot x = \mu+z^2+\textstyle{\frac{a}{32}}(\lambda -z) + (\textstyle{\frac{a}{64}})^2
\]
and so $\dot x \ge \mu+z^2$ if $\lambda >z$. In $A$, $z<\frac{1}{8}$
and so if
\begin{equation}\label{eq:ACcond}
\lambda > \textstyle{\frac{1}{8}}\quad \textrm{and} \quad \mu >-\frac{1}{64} .
\end{equation}
then $\dot x> 0$  as either
boundary is approached from within $A$, $\dot x=1$ in $C$, so the
consistent transversality condition is satisfied, and solutions pass transversely
across these boundaries too.

Fig.~\ref{fig:Cherry} illustrates the flow and return maps associated
with this model as $\mu$ passes through zero. The parameters used are
\begin{equation}\label{eq:cherrypars}
\lambda =1, \quad a=45, b=0.66\times \frac{3}{8}, \quad c=0.25, \quad \mu=\pm\frac{1}{70}.
\end{equation}
With these parameters we expect a gap of
size $2(\frac{1}{8\times 45})\approx  0.00555$ to open up as $\mu$
decreases through zero and increase in seize as $\mu$ decreases.

\begin{figure}
\centering{\includegraphics[scale=1.1]{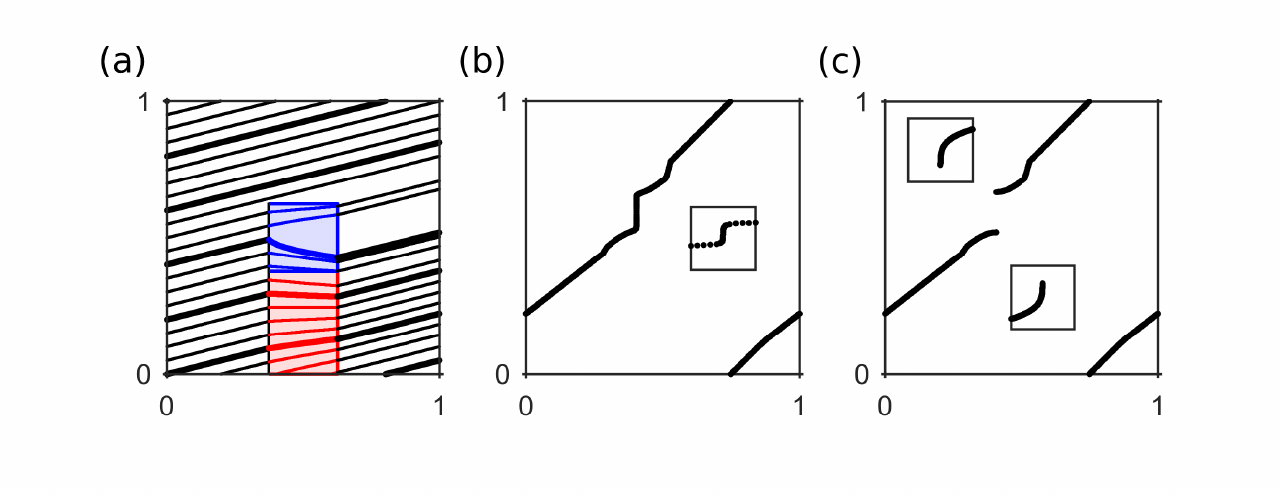}}
\caption{The piecewise smooth model of the transition to a Cherry
flow with parameters (\ref{eq:cherrypars}). (a) Flow with $\mu =\frac{1}{70}$ 
showing a stable solution that winds many times around the torus;
%(b) attractor for the flow with $\mu =-\frac{1}{70}$;
(b) return map on $x=0$ for $\mu=\frac{1}{70}$
with enlargement around the region with high derivative
($x_n\in[0.406245, 0.406255],
x_{n+1}\in[0.45, 0.7])$)
; and
(c) return map on $x=0$ for $\mu=-\frac{1}{70}$
with enlargements around each end of the discontinuity
($x_n\in[0.4061, 0.4064]$ with $x_{n+1}\in[0.6674, 0.6677]$
for the upper end, and $x_{n+1}\in[0.5198, 0.5201])$
for the lower end.)
}
\label{fig:Cherry}
\end{figure}

Figs~\ref{fig:Cherry}(a-b) show the flow and the associated
return map with $\mu=\frac{1}{70}$, i.e. when the flow is still a
Poincar\'e flow.
The return map clearly displays the very steep derivative over a
  significant region in the $x_{n+1}$ variable, but a very small
  region in the $x_n$ variable.
% Although the return map appears to have a gap,
% closer inspection shows that the slope does increase and there is
% evidence that the derivative becomes very large whilst the map remains
% continuous.
The inset shows a blow-up the map in a segment of $x_n$-values
containing the region of high derivative. This emphasises just how narrow the 
regions involved become and illustrates the rapid transition
  between shallow and steep derivatives.

Figs~\ref{fig:Cherry}(c) shows %the flow and
the associated return map with $\mu=-\frac{1}{70}$, at which the flow
is a Cherry flow.
For most $x_n$ values, the return map is almost identical, except
  that the steep curve has been replaced by a gap. The insets show a
  neighbourhood of the points of discontinuity. This reveals
%Although the return map and inset (on the same
%scale as before) strongly suggest a gap, the evidence for
the infinite slope at the point of discontinuity.
%is not visible at this resolution,
% even though it is one of the more standard parts of the analysis. Even
% at higher resolutions we have been unable to observe large derivatives
% at the discontinuity. We conjecture that this is because it is a very
% local phenomenon and the global contraction elsewhere required by the
% existence of the gap dominates outside a very small region of the
% discontinuity.

%%%%%%%%%%%%%%%%%%%%%%%%%%%%%%%%%%%%%%%%%%%%%%%%%%%%%%%%%%%%%%%%%%%%%%%%%%%%%%%%%%%%%%%%%
%Conclusion
%\input{Conc}
\section{Conclusion}
In this paper our focus has been to understand how structural
transitions occur in maps derived from fundamental %simple
models. We have considered transitions from continuity to
discontinuity, monotonicity to nonmonotonicity and the creation of
multiple gaps, and have described how these transitions can alter the bifurcations
and dynamics of circle maps. Understanding how these structural
transitions occur and their consequences suggests some new phenomena
and gives a wider context within which to interpret some of the
existing literature.

For example, in the study of maps with gaps much of the focus has
been on gaps where the derivatives at either side of the gap are
bounded.  Applications of such maps included threshold maps with
non-smooth thresholds, like the combs in~\cite{Glendinning_1995} or
the triangles and rectangles in~\cite{Arnold_1991}. However, these
non-smooth thresholds were introduced as approximations of smooth
thresholds to allow for explicit calculations and there are other
% . However, in some
applications (e.g.\ impact oscillators) in which
finite derivatives do not occur \cite{Avrutin_2010,BBCK}. Generically,
we have shown that both in threshold systems and in the creation of
Cherry flows one expects the associated discontinuous circle map to
have a singularity in the derivative to one (threshold models) or both
(Cherry flows) sides of the gap.  In the case of threshold systems, it
is the contact between the up/down flow and the upper/lower thresholds
that is important in determining the local behaviour: if the contact
is at a tangency, then generically a gap and the square root
singularity results. At the first tangency in a family of such maps
the size of the gap increases continuously from zero for the threshold
models, but for Cherry flow there is a discontinuous jump to a finite
gap at the transition point.

We have shown that the natural consequence of the square root
singularity is that one expects to see sequences of border collisions
(types I and II of section~\ref{sect:squareroot}) 
interspersed with saddle-node bifurcations. Using
a specific example threshold model we have illustrated how the Arnold
tongue bifurcation set for continuous monotonic circle maps with
periodic solutions created and destroyed by saddle-node bifurcations
transitions to a bifurcation set where periodic solutions can in
addition be created/destroyed by border collisions.  This underlying
structure underpins the bifurcation sets found numerically by Glass et
al~\cite{Glass_1991,Glass_1979} and the recent work on the two
process model for sleep-wake regulation \cite{Bailey_2018}.

The transition from no gaps to gaps in piecewise smooth monotonic maps
also has an important consequence for non-periodic solutions. With no
gaps, non-periodic solutions are quasiperiodic and typically they are dense in the
circle. With gaps, solutions tend to a Cantor set
\cite{Rhodes_1986}. Once noted, this difference is readily
observable in numerically computed bifurcation diagrams, as
illustrated in Fig.~\ref{fig:bifdiagnogapandgap}(b) and (c).

For threshold systems, we have identified that the transition to
nonmonotonicity is also the result of tangency, this time of the
up/down flow with the lower/upper threshold. The presence of both gaps
and nonmonotonicity gives many different new possibilities. For
example, we have shown that there is a natural transition from circle
maps with a single gap to multiple gaps.  This in turn leads to a
novel codimension two point in which there is the coincidence of two
border collisions.  A provisional analysis of this
codimension two point shows how it acts as a local organising centre, out
of which an infinite sequence of other border collisions emerge (cf. \cite{Granados_2017}); details will be published elsewhere.

Even the simple example model that we have chosen to illustrate many
of our results, the STS, has extremely rich dynamics which we have not
classified exhaustively and which will be the subject of future
work.  Our aim has been more to understand the structure
of some specific novel generic situations and provide an overall
framework.

\section*{Acknowledgements}
The authors would like to thank the intensive programme on Advances in
Nonsmooth Dynamics at the Centre de Recerca Matem\'{a}tica in Spring
2016. Discussions at this programme gave an important impetus to this work.
ACS would like to thank Leon Glass for interesting discussions on the
background of the STS model at the Biological Oscillator meeting at
the European Molecular Biology Laboratory in 2018.

\section*{References}

\bibliographystyle{siamplain}
\bibliography{sleep,maps}

\end{document}